\documentclass[leqno]{article}
\usepackage{blindtext}
\usepackage{fullpage}
\usepackage{amsmath}
\usepackage{amsthm}
\usepackage{amsfonts}
\usepackage{amssymb}
\usepackage{amssymb}
\usepackage{cases}
\usepackage{comment}
\usepackage{graphicx} 
\usepackage{mathtools}
\usepackage{soul}
\usepackage{mathabx}
\usepackage{xcolor}
\usepackage[hidelinks]{hyperref}
\usepackage{bbm}

\newtheorem{assumption}{Assumption}

\newtheorem{theorem}{Theorem}[section]
\newtheorem{notation}[theorem]{Notation}
\newtheorem{remark}[theorem]{Remark}
\newtheorem{definition}[theorem]{Definition}
\newtheorem{proposition}[theorem]{Proposition}
\newtheorem{lemma}[theorem]{Lemma} 
 
\newtheorem{corollary}[theorem]{Corollary}

\numberwithin{equation}{section}

\definecolor{LB}{rgb}{0,0.5,0.8}

\newcommand{\g}{\gamma}
\newcommand{\al}{\alpha}
\newcommand{\N}{\mathbb N}
\newcommand{\be}{\beta}
\newcommand{\R}{\mathbb R}

\newcommand{\E}{\mathbb E}

\def\shc{{\cal C}}
\def\shd{{\cal D}}

\def\shf{{\cal F}}

\def\shi{{\cal I}}

\def\shl{{\cal L}}
\def\shp{{\cal P}}
\def\shs{{\cal S}}

\def\P{\mathbb P}

\title{A non-local singular non-linear   Fokker-Planck PDE}
\author{Luca Bondi\footnote{University of Torino and ENSTA Paris, luca.bondi@unito.it}, Elena Issoglio\footnote{University of Torino, elena.issoglio@unito.it} and Francesco Russo\footnote{ENSTA Paris, francesco.russo@ensta.fr}}
\date{May 2026}

\begin{document}

\maketitle
\begin{abstract}
  The focus of this paper is a non-local singular non-linear Fokker-Planck partial differential equation (PDE). 
  The peculiarity of this PDE feature is in its divergence coefficient, which presents a product between a Besov distribution  and a non-linearity. The latter involves the convolution between an integrable kernel $K$ and the solution of the  PDE, which  leads to a non-locality of the first order term in the PDE.
  We prove existence and uniqueness of a solution to the PDE as well as continuity results on its coefficients.
  Previous analytical results are then applied
  to the study of  well-posedness in law for a 
  non-local singular McKean stochastic differential equation.
  As byproduct of that probabilistic representation, we establish mass conservation and positivity preserving
  for the PDE.
  \end{abstract}



\section{Introduction}\label{sc:intro}

In this paper we show the well-posedness   of a
 non-local singular non-linear Fokker-Planck partial differential equation (PDE) of the form
\begin{equation}\label{eq:FPmod}
    \left\{
     \begin{array}{l}
     \partial_t v=\frac12\Delta v-\text{div}(vF(K*v)b)\\
      v(0)=v_0,
     \end{array}
      \right.
    \end{equation}
    where $b$ is a function of time taking values in a $d$-dimensional Besov space  of negative index $-\beta<0$, hence $b$ is a distribution in the space variable, formally $b:[0,T]\times \R^d \to \R^d$, $K :\R^d \to \R$ and $v_0:\R^d \to \R$ are integrable functions which are also elements of an inductive Besov space of positive index $\beta+$ (see \eqref{eq:indBesov} below), and $F:\R \to \R$ is a non-linearity, see Definition \ref{def:besov} and Assumptions \ref{ass:beta}, \ref{ass:K}, \ref{ass:F} for more details.
For the PDE \eqref{eq:FPmod}, we also investigate the mass conservation and positivity preserving properties.
    
This paper also concerns the well-posedness of closely related non-local singular McKean SDEs of the type
\begin{equation}\label{eq:SDEmod}
\left\{
\begin{array}{l}
dX_t=F(K*v(t,X_t))b(t,X_t)dt+dW_t\\
v(t,\cdot)\text{ is the law density of } X_t\\
  X_0\sim v_0(x) dx,
\end{array}
\right.
\end{equation}
where $W$ is a $d$-dimensional Brownian motion.
We highlight  that equation \eqref{eq:SDEmod} is only formal at this level, because the term $b(t,X_t)$ cannot be evaluated,
since $b$ is a distribution. 

The Fokker-Planck denomination was formerly reserved  for linear PDEs, but nowadays it is also used in non-linear settings.
Fokker-Planck PDEs are generally related to stochastic processes, and at the analytical level they are
second order parabolic PDEs with mass conservation, whose solutions preserve positivity if the initial value is positive. For a recent survey on mass conservation in non-linear diffusions, see \cite{Vazquez23survey}. 
Maybe the most celebrated (local)  non-linear Fokker-Planck PDE is the so called generalised porous media equation, first investigated  by \cite{Benilan, BenilanCrandall, BrezisCrandall}. A more recent monograph at the analytical level on the topic is  \cite{RoecknerFokkerBook}.
One important contribution still in the linear local Fokker-Planck PDE was
provided by \cite{Jordan96}, which described 
the Fokker–Planck dynamics as a gradient flow of its own free energy.
The authors linked a solution of a linear Fokker-Planck PDE to  a minimiser
of a free-energy functional under the 2-Wasserstein distance for probability measures.
Moreover they formulated
 the so called
Jordan-Kinderlehrer-Otto (JKO) scheme, \cite{Jordan96}.
A constrained version of the JKO scheme was introduced in \cite{Carlen_Gangbo03} to deal with
non-linear  non-conservative (local) Fokker-Planck PDEs, which, even though 
they  dissipate the mass, they satisfy some conservation laws.
 
In the literature, we also find extensions to Fokker-Planck equations involving {\it non-local} terms,
and thus giving rise to integro Fokker-Planck PDEs. Non-locality may arise in applications at an intrinsic level, such as
in mathematical biology (swarm of insects, flock of birds, colony of bacteria) where individuals can interact not only
with their immediate neighbours but also with individuals far away, see \cite{swarm96a,swarm99,swarm96b} for
some examples on integro-differential population models.
 Furthermore, coefficients in any applied model are
often found to be non-local because measurements are often local averages. 

The JKO scheme mentioned earlier in the local case, opened the way to
investigate also non-local  PDEs.  
In \cite{tudorascu11} the authors study a non-linear parabolic problem of porous media type, featuring
a non-local first-order term proving existence of a weak solution that conserves mass, mean and variance.
The strategy of  their proof is based on the constrained (JKO) scheme of  \cite{Carlen_Gangbo03}.
 More recent works in this spirit are \cite{Eberle} and \cite{Ferreira}.  
 The non-local feature can be also found in the leading operator,  such as a fractional Laplacian, see e.g.\ Section 3 of \cite{GomezCastro24}
 which includes relevant contributions of L. A. Caffarelli and J. L. Vazquez.
We also mention  \cite{chaudru_jabir_menozzi22}, in relation to some probabilistic representations in the fractional Laplacian framework.
Nevertheless  these are beyond the scope of this paper.

Our Fokker-Planck PDE presents two
aspects that make its study challenging.

{\em (i)} First of all, the product between $vF(K*v)$ and  $b$ is defined as an element of a  Besov space of negative index, making use of the notion of  pointwise product between a function and a distribution, see \cite{bony}. 
Notice that the map which associates a generic probability
measure $\nu$ to $\nu F(K\ast \nu) b$ is not defined
because of the pointwise product with the distribution $b$, hence  it is not possible to use techniques involving 
   Wasserstein distance  for the study of  PDE \eqref{eq:FPmod}, hence JKO schemes are not available.

   {\em (ii)} Secondly, the PDE is non-local in the divergence term because of the convolutional term $K*v$, and
  so we cannot easily make use of PDE techniques
  in showing that the map
  \begin{equation}\label{eq:phi-intro}
v \mapsto v F(K\ast v) =: \phi(v),
\end{equation}
going from a Besov space of positive index $\beta$ into itself,
has linear growth, which for instance happens if $K = \delta_0$,
because $v \mapsto v F(v)$ is Lipschitz,
see Lemma 3.1 of \cite{issoglio_russoMK}.

Finally we highlight that PDE \eqref{eq:FPmod} is non-linear due to the non-linearity $F$ evaluated at $K*v$, which may be problematic
for any fixed point argument.
Below we will explain how we overcome those obstacles, in particular
the lack of linear growth property for the map
\eqref{eq:phi-intro}.

Our proposed approach
exploits the a priori probabilistic constraints that a solution 
$v$ {\em should} satisfy. It is well-known that  solutions to linear Fokker-Planck equations are classically
found as marginal laws of stochastic processes satisfying certain stochastic dynamics.
When those laws admit a density,
they have to be non-negative and their $L^1$-norm is conserved and equals 1.
First of all we show in Proposition \ref{pr:C_al}  that the map $\phi$ defined in  \eqref{eq:phi-intro} maps  $\mathcal C^\beta \cap L^1$ into itself
and has linear growth and locally Lipschitz property in
$\shc^\beta$ on the unit ball of $L^1$.
The notion of solution that we use for the fixed point is the mild one, namely a function $v\in C_T\mathcal C^{\beta+}$ such that 
 \begin{equation}\label{eq:mild-intro}
v(t)= P_{t}v_0-\int_0^t P_{t-s}\big({\rm div}\left(
\phi(v(s))
b(s)\right)\big)ds,
 \end{equation}
 where $(P_t)$ is the heat semigroup, see Definition \ref{def:MsolFP}. This choice 
is justified, since we show  the equivalence between mild solutions and  weak solutions,  see Definition \ref{def:WsolFP} and Lemma \ref{lm:W-M}.

 In our approach, the solution map
 is not, as usual, the right-hand side of \eqref{eq:mild-intro}.
Instead, for  a given  function $w \in C_T \shc^\beta$,
we associate the solution
$v = \tau (w)$ of the linear PDE
\begin{equation}\label{eq:FPlin-intro}
 \left\{
 \begin{array}{l}
   \partial_t v=\frac12\Delta v-{\rm div}(v g_w)\\
    v(0)=v_0,
    \end{array}
    \right.
\end{equation}
where  $v_0\in\shc^{\be+}$ and $g_w = F(K\ast w) b$, which belongs 
to $C_T\shc^{(-\be)+}$, see Notation \ref{not:gw}. The map
$\tau: C_T \shc^\beta \rightarrow  C_T \shc^\beta  $ is the new (linearised) solution map for \eqref{eq:FPmod}.
Indeed we will  investigate the linear PDE \eqref{eq:FPlin-intro}
with generic $g \in C_T\shc^{(-\be)+}$.
In particular, we show  its  well-posedness in Theorem \ref{thm:exunlinFP},
and  we prove
that the solution $v$ is bounded in $L^1$ and it 
is non-negative  provided that $v_0 \in L^1$ and it is non-negative, see Corollary  \ref{cor:lin_sing_PDE}.
Consequently, under previous assumption on $v_0$, $\tau$ takes values
in  $C_T\shc^{\be} \cap B_TL^1$, where $B_TL^1$ is the space of bounded $L^1$-valued functions of time.
As a consequence of Corollary  \ref{cor:lin_sing_PDE} and the generalized Gronwall lemma,
Proposition \ref{pr:uniFPlim} shows that \eqref{eq:FPmod} admits at most one solution,
in particular one gets uniqueness.

Suppose moreover that $ \Vert v_0 \Vert_{L^1} \le 1$. By Proposition \ref{lin_sing_PDE}, which is based on extensions of Feynman-Kac techniques, we know that
any given solution belongs to $S_{L^1}$, i.e. the space of non-negative functions $u$ defined on $[0,T]$
taking values in $L^1$-unit ball,
see \eqref{def:B1}. 
At this point by restriction, $\tau$ can be considered
%
as
\begin{equation}\label{eq:tau-intro}
\tau: C_T\shc^\al\cap S_{L^1}\to C_T\shc^\al\cap S_{L^1},
\end{equation}
which is shown to be a contraction via
 Proposition \ref{pr:boundtau}. Hence the unique fixed point of $\tau$ must be the unique
solution of the non-linear Fokker-Planck PDE \eqref{eq:FPmod}, see Theorem  \ref{thm:WEFPPDE}.
Beyond the well-posedness of the Fokker-Planck PDE \eqref{eq:FPmod}, we also establish its continuity property with
respect to the coefficient $b$.

Below we will discuss the probabilistic representation of  \eqref{eq:FPmod}.
We emphasize that, using that tool, Theorem \ref{thm:lin-lim} allows to prove
mass conservation and positivity preserving of the solution of
\eqref{eq:FPmod}. We point out that this result holds even when the kernel $K$
is any probability measure.

\vspace{10pt}
Let us now give some details about the link between Fokker-Planck PDEs and stochastic processes.
Linear Fokker-Planck PDEs, as foretold above, arise naturally when  studying the time evolution of the density of a class of
stochastic processes. 

On the other hand, given a solution $v:[0,T] \rightarrow \shp(\R^d)$,
of a Fokker-Planck PDE with bounded measurable coefficients
one can construct stochastic process, solution of
a martingale problem, whose marginal laws are $v$.
This process is the so called {\em probabilistic representation}
of the Fokker-Planck PDE. This was done in 
Theorem 2.6 in \cite{figalli} and it was generalised in  Theorem 2.5 in \cite{Trevisan}
and Theorem 1.1 of \cite{bogachev_superposition}.
This link can be shown to hold under suitable conditions  also in the case of non-linear Fokker-Planck PDEs.
The starting point was the seminal work  by McKean \cite{McKean} in the mid 60's,
which introduced the so called McKean SDEs, which are SDEs, whose coefficients  depend not only
on time and space, but also
on the law of the solution itself.
More recently, in the case of irregular coefficients 
we refer e.g.
to
\cite{BRR, BRR2, BCR2} in the case of porous media and fast diffusion PDEs,  \cite{Olivrichtoma} in the case of Keller-Segel
models and
\cite{BarbuRoeckSuperposition, BarbuRockSIAM} in the general case, when the coefficients are still functions,
making use of the superposition principle.
Finally, in the case  the drift is a Schwartz distribution, we refer for instance to
\cite{issoglio_russoMK, issoglio_et.al24},
and \cite{BIR_Review} for a recent survey.

As anticipated, this paper concerns non-local singular McKean SDEs of the type \eqref{eq:SDEmod},
which constitutes the probabilistic representation of
the non-local singular non-linear 
Fokker-Planck PDE \eqref{eq:FPmod}.
Since $b(t,X_t)$ cannot be evaluated,
because  $b$ is a distribution, we must rely on a notion of solution for singular  SDEs  that is framed through the
{\em rough martingale problem} without law dependency, see Definition \ref{def:rMP}. The solution to a rough martingale problem is  a probability
measure on the canonical space which corresponds to the law of a stochastic process which is formally the solution
to the SDE  \eqref{eq:SDEmod}. Notice that \eqref{eq:SDEmod} is mathematically  meaningless
without  noise: this phenomenon is known in stochastic analysis as regularization by noise, see e.g.\ \cite{flandoli-book}.
Our main result  on well-posedness of \eqref{eq:SDEmod} is Theorem \ref{thm:WPMKSDE}, and the analysis on the non-local singular non-linear Fokker-Planck PDE \eqref{eq:FPmod}  is  key to its proof.
It turns out that, the solution of \eqref{eq:SDEmod} will be constructed via
a solution of an associated  rough martingale problem without law dependency.
For this purpose, in  Proposition \ref{prop:SDEtoPDE} we  show that the marginal laws of the solution to a rough martingale problem
in the sense of Definition \ref{def:rMP}, 
admit a density   and  such densities constitute  a solution to the corresponding linear Fokker-Planck PDE.

\vspace{10pt}

The paper is organised as follows. Section \ref{sc:prelim} introduces the definition of the Besov spaces employed
in this study,  useful estimates of the heat semigroup like Schauder's and Bernstein's estimates, an associativity property
for Besov spaces, some topological preliminaries and three important continuity results: Lemma \ref{lm:BNcont}, Lemma \ref{lm:prep0}
 and Proposition \ref{pr:C_al}.
 Section \ref{sc:FP} introduces the notions of solution (weak and mild) of \eqref{eq:FPmod} belonging to $C_T \shc^\beta$
 and of its linearised version.
 Theorem \ref{thm:lin-lim} states the a priori property of a solution to \eqref{eq:FPmod} in terms
 of mass conservation and positivity preserving.
 Section \ref{ssc:linearFP}  is dedicated to the study of the associated singular  linear Fokker-Planck  PDE
 \eqref{eq:FPlin-intro}, with its  well-posedness proved in Theorem \ref{thm:exunlinFP}.
 Useful continuity results  are stated
 in Theorem \ref{thm:exunlinFPCont}, and $L^1$-bounds  in  Proposition \ref{lin_sing_PDE} and Corollary \ref{cor:lin_sing_PDE}.
 Section \ref{ssc:nonlinear} contains the proof of the main result on  well-posedness of \eqref{eq:FPmod}, stated in
Proposition \ref{pr:uniFPlim} (uniqueness) and
 Theorem \ref{thm:WEFPPDE} (existence).
 Section \ref{sec:PRFP} is devoted to the probabilistic representation of   \eqref{eq:FPmod}. In particular, Section \ref{sc:rMP} introduces the tools and the definition of rough martingale problem, and Section \ref{ssc:MK} is dedicated to the study of well-posedness of the singular McKean SDE \eqref{eq:SDEmod} via such rough martingale problem.

\section{Topological setting }\label{sc:prelim}
 
\subsection{Basic definitions and recalls}

We give some definitions and preliminary results on spaces and functions which will be used later on. For a strictly positive integer $m$, $C_b^m(\R^d)$ (resp. $C_b^{1,m}([0,T] \times \R^d)$) is the Banach space of $C^m(\R^d)$ (resp. $C_b^{1,m}([0,T] \times \R^d)$) bounded
real-valued functions with bounded derivatives equipped with
the usual sup-norm, the latter denoted as $\|\cdot\|_{\infty}$. We denote the space of real-valued Schwartz functions on
$\R^d$ by $\shs(\R^d)=\shs$ and  the space of tempered distributions $\shs'(\R^d)=\shs'$. $T >0$ will be a fixed terminal time. 
We denote by $\mathcal{F}$ and $\mathcal{F}^{-1}$ the Fourier transform on $\shs$ and
its inverse, see Definition 1.20 of \cite{sawano}, which are extended on $\mathcal{S}'$ in the usual way, see Definition 1.21 of \cite{sawano}. The dual pairing on $\shs$ and $\shs'$ will be indicated with $\langle\cdot,\cdot\rangle$. $L^1:=L^1 (\R^d)$ will denote the space of real-valued
Lebesgue integrable functions, and $\shp(\R^d)$  denotes the set of all Borel  probability measures on $\R^d$.  Given a topological vector space $E$, $C_TE$ will denote the space of $E$-valued continuous functions defined on $[0,T]$, and, whenever $E$ is a Banach space, $B_TE$ will denote the space of $E$-valued bounded functions defined on $[0,T]$.

Throughout the paper $c$ will denote a universal positive constant, while $C$ will denote a positive  constant which will depend on one or more parameters  (possibly indicated with subscripts), that will be specified case by case. These constants can vary from line to line.
Given a function $f:[0,T] \times \R^d \rightarrow \R$, we will often denote $f(s):= f(s,\cdot)$.


Besov spaces are a fundamental element of this study. A classical general reference is \cite{sawano},
Section 2.1, which uses the denomination Nikolskii–Besov spaces.

\begin{definition}\label{def:besov}
For every $\gamma \in \R$, $\shc^{\g}$ will denote $B^{\g}_{\infty,\infty}(\R^d)$.
     Furthermore, we define the inductive Besov space $\shc^{\g+}$ as
      \begin{equation}\label{eq:indBesov}
      \shc^{\g+}:=\bigcup_{\al>\g}\shc^{\al}.
    \end{equation}
 We also define
  $\mathcal C^{\gamma-}$  the  topological space given by
\[
\mathcal C^{\gamma-}:= \bigcap_{\alpha <\gamma} \mathcal C^{\alpha},
\]
  see Section 2.1 of \cite{issoglio_russoPDEa}.

  \end{definition}
By Proposition 2.3 in \cite{sawano}, we remark that $\shc^{\g'}\subseteq \shc^{\g}$ for all $\g,\g'\in\R$ such that $\g<\g'$.
By Theorem 2.7 in  \cite{sawano}, we know that if $\g\in(0,1)$, then the Besov Space $\shc^\g$ is equivalent to the
H\"older-Zygmund space, which we will still denote $\shc^\g$ by a slight abuse of notation. This is endowed with the norm $\|\cdot\|_{\shc^\g}$, defined for all $h\in\shc^\g$ as
\begin{equation}\label{eq:holdnorm}
\|h\|_{\shc^\g}:=\|h\|_{\infty}+\|h\|_{\g}:=\sup_{x\in\R^d}|h(x)|+\sup_{x,y\in\R^d, x\not=y}\frac{|h(x)-h(y)|}{|x-y|^\g}.
\end{equation}
In the whole paper, for a given $\gamma \in \R$, we will still denote by $\shc^\gamma$ any vector whose elements
belong $\shc^\gamma$. 
Finally we recall that Lemma B.2 of \cite{issoglio_russoMK} guarantees that $C_T\shc^{\g+}=\bigcup_{\al>\g}C_T\shc^{\al}$. This constitutes a characterisation of these inductive spaces of distributions. 
In particular,
if a sequence $(h^n)$ in $C_T\shc^{\g+}$ converges to some $h\in C_T\shc^{\g+}$ then there exists $\alpha>\g$ such that $h^n \to h$ in $C_T\shc^{\alpha}$. Analogously for the convergence in $\shc^{\g+}$.

Next we list the assumptions that will be needed later on. 
\begin{assumption}\label{ass:beta} 
  $0<\be<\frac12$ and $ b$ is a $d$-dimensional vector, whose components belong to $C_T\shc^{(-\be)+}$.
 \end{assumption}
\begin{assumption}\label{ass:K}
    $K:\R^d\to \R$ is a probability density function which is also an element of $\shc^{\be+}$.
\end{assumption}
\begin{assumption}\label{ass:F}
    The function $ F\in C^1(\R)$ such that $ F'$ is bounded and Lipschitz.
\end{assumption}

\begin{definition}\label{def:heatkern}
  Let $(P_t)$ denote the semigroup generated by $\frac12\Delta$ on $\shs$, in particular for all $\Phi\in\shs$ we define
  $P_t \Phi = \int_{\R^d} p_t(x-y) \Phi(y) dy, t \ge 0$,
 where $p$ is the heat kernel
\begin{equation}\label{eq:heatkern}
p_t(z):=\frac1{(2\pi t)^{d/2}} \exp^{-\frac{|z|^2}t}.
\end{equation}
It is easy to see that $P_t:\shs\to\shs$. 
Moreover we can extend it to $\shs'$ by dual pairing (and we denote it with the same notation
for simplicity), namely one has $\langle P_t\Psi, \Phi\rangle:=\langle \Psi,P_t\Phi \rangle$
for all $\Psi\in\shs'$ and $\Phi\in\shs$, the equality holds without changing the sign, because $p$ is symmetric.
\end{definition}
The estimates below are known as \emph{Schauder's estimates}: for a proof we refer to \cite[Lemma 2.5]{catellier_chouk}, see also \cite{gubinelli_imkeller_perkowski} for similar results. 

\begin{lemma} \label{lm:schauder}
    Let $f\in \shc^{\g}$ for some $\gamma \in \R$. Then for any $\theta\geq 0$ there exists a constant $c$ such that
    \begin{equation}\label{eq:schauder1}
    \|P_tf\|_{\shc^{2\theta+\g}}\leq c t^{-\theta} \|f\|_{\shc^\g},        
    \end{equation}
    for all $t>0$. 
  Moreover for $f\in\shc^{2\theta+\g}$ and for all $\theta \in (0,1)$ we have
    \begin{equation}\label{eq:schauder2}
        \|P_tf-f\|_{\shc^\g}\leq ct^{\theta} \|f\|_{\shc^{2\theta+\g}}.
    \end{equation}
  \end{lemma}
 \medskip

  We also recall  Bernstein's inequalities, see \cite[Lemma 2.1]{bahouri} and \cite[Appendix A.1]{gubinelli_imkeller_perkowski}.
\begin{lemma} \label{lm:bern}
Let $\g\in\R$ then there exists $c>0$ such that 
\begin{equation*}
    \|\nabla g\|_{\shc^\g}\leq c\|g\|_{\shc^{\g+1}},
\end{equation*}
for all $g\in\shc^{\g+1}$.
\end{lemma}
Using Schauder's and Bernstein's inequalities we can easily obtain a useful estimate on the gradient of the semigroup, as we see below. 
\begin{corollary} \label{cor:bernschau}
Let $\g\in\R$ and $\theta\in (0,1)$. If $g\in\shc^{\g}$ then for all $t>0$ we have
$\nabla(P_tg) \in \shc^{\g+2\theta-1}$ and
\begin{equation}
    \|\nabla(P_tg)\|_{\shc^{\g+2\theta-1 }}\leq ct^{-\theta}\|g\|_{\shc^\g}.
\end{equation}
\end{corollary}
We formulate a slight straightforward adaptation of  Proposition 2.4 of \cite{issoglio_russoMK},  which will be used later.
\begin{lemma}\label{lm:Xncont}
Let $g\in C_T\shc^{(-\be)+}$.  Let us consider the sequence $(g_n)$ defined by
\begin{equation}\label{eq:molbn}
g_n(t,\cdot) :=p_{\frac1n} \ast g(t,\cdot), \ t \in[0,T ],
\end{equation}
 $n \geq1$, where $p$ was defined in \eqref{eq:heatkern}. Then we have the following.
\begin{itemize}
\item[(i)] For each $n$, $g_n$ is globally bounded, together with all its space derivatives.
\item[(ii)] For each $n$, $t\mapsto g_n(t,\cdot)$ is continuous in $\shc^\g$ for all $\g>0$. In particular $g_n\in C_T\shc^{(-\be)+}$.
   \item[(iii)] We have the convergence $g_n \to g $ in $C_T\shc^{(-\be)+}$.
\end{itemize}

\end{lemma}
The following is an important estimate, which allows to define the so-called {\em pointwise product} of distributions,
see  (2.9) and (2.10) in \cite{issoglio_russoMK}.  
Let $\g,\al>0$
with $\g-\al>0$.
There exists a constant $c>0$ such that
\begin{equation}\label{eq:bony}
    \|f g\|_{\shc^{-\al}} \leq c\|f\|_{\shc^{\g}}\|g\|_{\shc^{-\al}},
  \end{equation}
  for all $f \in \shc^{\g}$ and $g\in \shc^{-\al}$: in particular
  the pointwise product $fg$ is a well-defined element of $\shc^{-\al}$.
  Furthermore,
for all $f\in C_T\shc^{\g}$ and $g\in C_T\shc^{-\al}$, then $ f g \in C_T\shc^{-\al}$ and
\begin{equation}\label{eq:bonyt}
    \|f g\|_{C_T\shc^{-\al}} \leq c\|f\|_{C_T\shc^{\g}}\|g\|_{C_T\shc^{-\al}}. 
  \end{equation}

 \begin{remark} \label{rmk:associativity}
   \begin{enumerate}
   \item An immediate consequence of \eqref{eq:schauder2}
     is that, for every $\gamma \in \R$ the set
     $C^\infty \cap \shc^{\gamma +}$ is dense in
     $\shc^{\gamma +}$.
     \item Let $\beta > 0, \alpha \in (0,\beta)$.
        Let $ \Gamma \in \shc^{-\alpha}, l_1,l_2 \in \shc^{\be}.$
    We have the associativity property
    \begin{equation} \label{eq:ass}
      (\Gamma  \cdot l_1) \cdot l_2 = (\Gamma \cdot l_1) \cdot l_2,
      \end{equation}
    where $\cdot$ is the pointwise product.
    Indeed we denote
    $$ \Gamma^n:= P_{1/n} \Gamma, l_i^n= P_{1/n} l^i, i = 1,2.$$
    First we observe that, being $\Gamma^n, l_1^n, l_2^n$ smooth
    we have
\begin{equation} \label{eq:assReg}
  (\Gamma^n  \cdot l^n_1)  \cdot l^n_2 = (\Gamma^n \cdot l^n _1) \cdot l^n_2,
  \end{equation}
     since the pointwise product extends the usual product of smooth functions.
     We set $\varepsilon = \frac{\beta - \alpha}{4}$.  \eqref{eq:schauder2} 
     shows that $\Gamma^n \rightarrow \Gamma$ in $\shc^{- \alpha - \varepsilon} $
     $l_i^n \rightarrow l_i, i = 1,2$, in $\shc^{\beta - \varepsilon} $.
     Taking
     into account successively \eqref{eq:bony}, we can show the
     convergence of the left-hand (resp. right-hand) side of 
     \eqref{eq:assReg} to the left-hand (resp. right-hand) side of 
     \eqref{eq:ass}.

   \end{enumerate}
  \end{remark}
  
  The same proof as in Lemma 3.2 of \cite{issoglio19} allows to establish the following. 

  \begin{lemma} \label{lm:issoglio}
If  $h\in B_T \mathcal C^{(-\beta - 1)+}$ then $\int_0^\cdot P_{\cdot-s} h(s) ds \in C_T \mathcal C^{\alpha}$ with any $\alpha < 1-\beta$.
\end{lemma}
\begin{proof}
  The proof  follows by Lemma 3.2 of \cite{issoglio19}, by re-doing all the steps with $\be$ replaced $-\be-1$ and $\al+1$ replaced with $\al$. The time-singularity arising from Schauder estimate \eqref{eq:schauder1}  depends on the difference of the
  regularity indices, which is $\al+ 1 - (-\beta)$ in the original proof, and it is  $\al  - (-\be-1)$ in this case, so they both are equal to  $\al +\be+1$.
\end{proof}

We introduce now a family of equivalent norms for the space $C_T\shc^\gamma$ when $\gamma \in \R$.

\begin{definition}\label{def:rhonorm}
    For all $\gamma\in \R$ and $\rho \geq 0$ we define the  $\rho$-norm
    \begin{equation}
        \|w\|^{(\rho)}_{C_T\shc^{\gamma}}:=\sup_{t\in [0,T]}e^{-\rho t}\|w(t)\|_{\shc^{\gamma}}.
    \end{equation}
 In the sequel, for notational reasons, we will set
    \begin{equation}\label{rho_dist}
        d_{\rho,\g}(w,z):=\|w-z\|^{(\rho)}_{C_T\shc^\gamma}.
    \end{equation}
 \end{definition}

We will denote with $S_{L^1}$ the set
 \begin{equation}\label{def:B1}
  S_{L^1}:=\left\{u:[0,T]\to L^1(\R^d)\text{ s.t. }\forall\ t\in[0,T],\ u(t,\cdot) \ge 0 \  {\rm a.e.}, \|u(t)\|_{L^1}\leq 1\right\},
   \end{equation}
 namely a set  of non-negative functions on $[0,T]$ taking values on the unitary ball in $L^1(\R^d)$.

\begin{lemma}\label{lm:B1close}
    The space $C_T\shc^{\gamma}\cap S_{L^1}$ is closed in $C_T\shc^{\gamma} $ for $\gamma\in(0,1)$.
    \begin{proof}
      Let $(u_n)\subseteq C_T\shc^{\gamma}\cap S_{L^1}$
be a sequence converging to $u \in C_T\shc^{\gamma}$.
It is obviously non-negative being a pointwise limit of non-negative functions.
      We are left to prove that for all $t\in [0,T]$
        \begin{equation*}
            \|u(t)\|_{L^1}=\int_{\R^d} u(t,x) dx \leq 1.
        \end{equation*}
Using Fatou's lemma we have for all $t\in[0,T]$
\begin{equation*}
    \|u(t)\|_{L^1}=\int_{\R^d}\liminf_{n\to\infty} u_n(t,x) dx\leq \liminf_{n\to\infty}\int_{\R^d} u_n(t,x) dx=1,
\end{equation*}
which concludes the proof.
    \end{proof}
  \end{lemma}
  
\subsection{Useful continuity results}

In this section we prove some useful bounds related to the non-linear non-local  term that appears in the PDE \eqref{eq:FPmod}.
\begin{lemma}\label{lm:prep}
	Let $\gamma \in (0,1)$ and $K\in \shp(\R^d)$.  Let us suppose Assumption \ref{ass:F}. Then we have the following.
	\begin{itemize}
		\item[(i)] $K*f\in\shc^\gamma$ and
		\begin{equation}\label{eq:Kstarv}
			\|K*f\|_{\shc^\gamma}\leq \|f\|_{\shc^\gamma},  \forall f \in\shc^\gamma.  
		\end{equation}
		\item[(ii)] For all $\ell_1,\ell_2\in\shc^\gamma,  F(\ell_i)\in\shc^\gamma$ for $i=1,2$. Moreover we have
		\begin{equation}\label{eq:LipF}
			\|F(\ell_1)-F(\ell_2)\|_{\shc^\gamma}\leq C_F\|\ell_1-\ell_2\|_{\shc^\gamma}
		\end{equation}
		and
			\begin{equation}\label{eq:Fl_al}
			\|F(\ell_1)\|_{\shc^\gamma}\leq C_F(1+\|\ell_1\|_{\shc^\gamma}).
		\end{equation}
	\end{itemize}
\end{lemma}
\begin{proof}
	 Let us fix $f\in\shc^\gamma$. 
	We start by proving that 
	$K*f$ is an element of $\shc^\gamma$. From direct calculations, using \eqref{eq:holdnorm} we get 
	\begin{align*}
		\|K*f\|_{\shc^\gamma}=&\|K*f\|_{\infty}+\|K*f\|_{\gamma}\\
		\notag
		=&\sup_{x\in\R^d}\left| \int_{\R^d} f(x-z)K(dz) \right|+\sup_{\underset{x\not=y}{x,y\in\R}}\frac{|K*f(x)-K*f(y)|}{|x-y|^\gamma}\\
		\notag
		\leq&  \int_{\R^d} K(dz) \|f\|_{\infty}+\sup_{\underset{x\not=y}{x,y\in\R^d}}\frac{\left| \int_{\R^d} [f(x-z)-f(y-z)]K(dz)\right|}{|x-y|^\gamma}\\
		\notag
		\leq&   \|f\|_{\infty}+ \int_{\R^d} \sup_{\underset{x\not=y}{x,y\in\R^d}}\frac{\left|f(x-z)-f(y-z)\right|}{|x-y|^\gamma}K(dz)\\
		\notag
		=&\|f\|_{\infty}+ \|f\|_{\gamma}=\|f\|_{\shc^\gamma}.
	\end{align*}
	Next we prove \eqref{eq:LipF}, making use again of \eqref{eq:holdnorm} for the norm of the difference between $F(\ell_1)$ and $F(\ell_2)$ and differentiability assumption on $F$.
	\begin{align*}
		\|F(\ell_1)-F(\ell_2)\|_{\shc^\gamma}=&	\|F(\ell_1)-F(\ell_2)\|_{\infty}+	\|F(\ell_1)-F(\ell_2)\|_{\gamma}\\
		\leq & \|F'\|_{\infty}(\|\ell_1-\ell_2\|_{\infty}+\|\ell_1-\ell_2\|_{\gamma})\\
		=&\|F'\|_{\infty}\|\ell_1-\ell_2\|_{\shc^\gamma}.
	\end{align*} 
	In particular, the latter bound, together with the triangle inequality,
 implies  that  $F (\ell_1) \in \shc^\gamma$ and \eqref{eq:Fl_al}.
 
\end{proof}
\begin{corollary}\label{cr:prep}
  Let suppose Assumption \ref{ass:F}  and $K\in\shp(\R^d)$. The following holds. 
	\begin{itemize}
		\item[(i)] $K*f\in C_T\shc^\gamma$ and
	\begin{equation}\label{eq:Kstarvt}
	\|K*f\|_{C_T\shc^\gamma}\leq \|f\|_{C_T\shc^\gamma},
      \end{equation}
      for every $f\in C_T\shc^\gamma$.
\item[(ii)] $F(K*f)\in C_T\shc^\gamma$ and 
\begin{equation}\label{eq:Fl_alt}
	\|F(K*f)\|_{C_T\shc^\gamma}\leq C_F(1+\|f\|_{C_T\shc^\gamma}),  \forall f \in C_T\shc^\gamma.
\end{equation}
\end{itemize}
\end{corollary}
\begin{proof}
The map $t\mapsto K*f(t)$ takes values in $\shc^\gamma$ from {\it item (i)} of Lemma \ref{lm:prep} and it is continuous
because   
 for $0\leq s\leq t\leq T$, we obtain
	\begin{equation*}
		\|K*f(t)-K*f(s)\|_{\shc^\gamma}=\|K*(f(t)-f(s))\|_{\shc^\gamma}\leq C_F\|f(t)-f(s)\|_{\shc^\gamma},
              \end{equation*}
from linearity of convolution and \eqref{eq:Kstarv} with $f=f(t)-f(s)$.
Now \eqref{eq:Kstarvt} follows now directly from  \eqref{eq:Kstarv} taking the supremum in time.
 $F(K*f)\in C_T\shc^\gamma$, since
 for every $s,t\in[0,T]$, we can use \eqref{eq:LipF} with $\ell_1=K*f(t)$ and $\ell_2=K*f(s)$ and \eqref{eq:Kstarv} with $f = f(t)-f(s)$ one gets
	\begin{equation*}
		\|F(K*f(t))-F(K*f(s))\|_{\shc^\gamma}\leq C_F\|K*f(t)-K*f(s)\|_{\shc^\gamma}\leq C_F\|f(t)-f(s)\|_{\shc^\gamma},
	\end{equation*}
        and continuity is proven. 
        The inequality \eqref{eq:Fl_alt} follows from \eqref{eq:Fl_al} with $\ell_1=K * f(t)$ for all $t\in[0,T]$ and \eqref{eq:Kstarv}
        with $f = f(t)$, taking finally the supremum 
  for  $t \in [0,T]$.
\end{proof}

\begin{corollary} \label{lm:BNcont}
	Let Assumption \ref{ass:F} hold and let $K\in \shp(\R^d)$. If  $f,g \in C_T \shc^\gamma$ for some $\gamma \in(0,1)$, then $gF(K*f)\in C_T\shc^\gamma$. In particular
	\begin{equation}
		\|gF(K*f)\|_{C_T\shc^\gamma}\leq  \|g\|_{C_T\shc^\gamma}C_F(1+ \|f\|_{C_T\shc^\gamma}).
	\end{equation}
\end{corollary}

\begin{proof}
  Thanks {\it item (ii)} of Corollary \ref{cr:prep}  we have that $F(K*f) \in C_T\shc^\gamma$ and so $gF(K*f) \in C_T\shc^\gamma$
  since $C_T\shc^\gamma$ is an algebra.
Using
Corollary 2.86 of \cite{bahouri}, with $s= \gamma, p = r = \infty$,
states
	\begin{equation}\label{eq:bahouri}
		\| g F (K*f) \|_{\shc^\gamma} \leq c \|g\|_{\shc^\gamma} \|F (K*f)\|_{\shc^\gamma}.
	\end{equation}
This directly implies that,
for each    $t\in[0,T]$,
one has
	\begin{equation*}
		\|g(t)F(K*f(t))\|_{\shc^\gamma}\leq  \|g(t)\|_{\shc^\gamma}C_F(1+ \|f(t)\|_{\shc^\gamma}) ,
	\end{equation*}
	and taking the sup over $[0,T]$ we can conclude. 
      \end{proof}
      
\begin{corollary}\label{lm:prep0bis}
	 Let $\gamma \in (0,1)$.
	If $f,g\in\shc^\gamma$,   $K\in\shp(\R^d)$ and Assumption \ref{ass:F} hold, then $gF(K*f)\in \shc^\gamma$ and 
	\begin{equation}\label{eq:fg-alpha}
		\|gF(K*f)\|_{\shc^\gamma} \leq  \|g\|_{\shc^\gamma} C_{F}(1+\|f\|_{\shc^\gamma}).
	\end{equation}
\end{corollary}
At this point we establish further estimates supposing Assumption \ref{ass:K},
in term of the $L^1$-norm of $f$.
\begin{lemma}\label{lm:prep0}
  Let $\gamma \in (0,1)$ and Assumptions \ref{ass:K} and \ref{ass:F} hold. Let $g\in\shc^\gamma$ and $f\in \shc^\gamma\cap L^1$. Then we have 
\begin{equation}\label{eq:fg-L1}
\|gF(K*f)\|_{\shc^\gamma}\leq \|g\|_{\shc^\gamma}C_{F,K}(1+\|f\|_{L^1}).
\end{equation}
\end{lemma}

\begin{proof}

We already know that $gF(K*f)\in\shc^\gamma$ from Corollary \ref{lm:prep0bis}. For $f\in\shc^\gamma\cap L^1$, through direct calculations,
taking into account \eqref{eq:bahouri}, 
  we obtain
 \begin{align*}
\|gF(K*f)\|_{\shc^\gamma} \leq& c\|g\|_{\shc^\gamma}\|F(K*f)\|_{\shc^\gamma}\\
=&c\|g\|_{\shc^\gamma} \left(\|F(K*f)\|_{\infty}+ \sup_{\underset{x\not=y}{x,y\in\R^d}}\frac{|F(K*f(x))-F(K*f(y))|}{|x-y|^\gamma}\right)\\
\leq&c\|g\|_{\shc^\gamma} \left(\vert F(0) \vert+\|F'\|_{\infty}\|K*f\|_{\infty}+ \|F'\|_{\infty} \sup_{\underset{x\not=y}{x,y\in\R^d}}\frac{|K*f(x)-K*f(y)|}{|x-y|^\gamma}\right)\\
\leq & \|g\|_{\shc^\gamma}\left(\vert F(0) \vert+\|F'\|_{\infty}\|K*f\|_{\infty}+\|F'\|_{\infty} \sup_{\underset{x\not=y}{x,y\in\R^d}}\int_{\R^d}\frac{|K(x-z)-K(y-z)|}{|x-y|^\gamma}|f(z)|dz \right)\\
\leq & c\|g\|_{\shc^\gamma}\left(\vert F(0) \vert+\|F'\|_{\infty}\|K\|_{\infty}\|f\|_{L^1}+\|F'\|_{\infty}\|K\|_{\gamma}\|f\|_{L^1}\right)\\
= & c \|g\|_{\shc^\gamma}\left(\vert F(0) \vert+\|F'\|_{\infty}\|K\|_{\shc^\gamma}\|f\|_{L^1}\right)\\
\leq & c \|g\|_{\shc^\gamma}C_{F,K}\left(1+\|f\|_{L^1}\right).\qedhere
\end{align*}

\end{proof}


\begin{lemma}\label{lm:prep2}
	Let $\g\in(0,1)$ and Assumptions \ref{ass:K} and \ref{ass:F} hold. Let also $f,g\in \shc^\gamma\cap L^1$. Then we have 
	\begin{equation}\label{eq:FC_al}
		\|F(K*f)-F(K*g)\|_{\shc^\gamma}\leq C_{F,K}(1+\|f+g\|_{L^1})\|f-g\|_{\shc^\gamma}.
	\end{equation}
\end{lemma}
\begin{proof}
In order  to estimate  $\|F(K*f)-F(K*g)\|_{\shc^\gamma}$,
we first observe that
	\begin{equation} \label{eq:I2infty}
		\|F(K*f)-F(K*g)\|_{\infty} \le \Vert F' \Vert_\infty \Vert K*f-K*g\Vert_\infty \le \Vert F' \Vert_\infty \Vert f - g\Vert_\infty, 
	\end{equation}
	so that it remains to estimate the $\Vert \cdot \Vert_\gamma$-seminorm, in \eqref{eq:holdnorm}. We proceed by estimating its numerator.
	\begin{align*}
		&|F(K*f(x))-F(K*f(y))-F(K*g(x))+F(K*g(y))|\\
		\leq &\bigg\vert\int_0^1 F'(aK*f(x)+(1-a)K*g(x))da (K*(f-g)(x))
		-\int_0^1 F'(aK*f(y)+(1-a)K*g(y))da (K*(f-g)(y))\bigg\vert\\
		\leq& \|F'\|_{\infty}\vert K*(f-g)(x)-K*(f-g)(y)\vert+\frac12\|F''\|_{\infty}\vert K*(f-g)(y)\vert \vert (K*f(x)-K*f(y) + K*g(x)-K*g(y) )\vert \\
		\leq& \|F'\|_{\infty}\|K*(f-g)\|_\gamma |x-y|^\gamma+\frac12\|F''\|_{\infty}\| K*(f-g)\|_{\infty}  \|K \ast (f+g)\|_{\gamma} |x-y|^\gamma \\
		\leq& \|F'\|_{\infty}\|K\|_{L^1}  \|f-g\|_{\gamma}|x-y|^\gamma+\frac12\|F''\|_{\infty}\| K\|_{L^1}\|f-g\|_{\infty}  \|K \|_{\gamma} \|f+g\|_{L^1}  |x-y|^\gamma\\
		\leq &\big(\|F'\|_{\infty}+\frac12\|F''\|_{\infty}\|K\|_{\gamma}\|f+g\|_{L^1}\big)\|f-g\|_{\shc^\gamma}\vert x-y\vert^\gamma\\
		\leq &\left(\|F'\|_{\infty}+\frac12\|F''\|_{\infty}\|K\|_{\gamma}\right)(1+\|f+g\|_{L^1})\|f-g\|_{\shc^\gamma}\vert x-y\vert^\gamma.
	\end{align*}
	Thus, dividing by $|x-y|^\gamma$ and taking the supremum over $x,y\in\R^d$, and combining the above calculations with
	\eqref{eq:I2infty}, it holds that
	\begin{align*}
		\notag  \|F(K*f)-F(K*g)\|_{\shc^\gamma}\leq &\|F'\|_{\infty}\|f-g\|_{\infty} +\left(\|F'\|_{\infty}+\frac12\|F''\|_{\infty}\|K\|_{\gamma}\right)(1+\|f+g\|_{L^1})\|f-g\|_{\shc^\gamma}\\
		\notag \leq & \left(\|F'\|_{\infty}+\left(\|F'\|_{\infty}+\frac12\|F''\|_{\infty}\|K\|_{\gamma})(1+\|f+g\|_{L^1}\right)\right)\|f-g\|_{\shc^\gamma}\\
		\leq & \left(2 \|F'\|_{\infty}+\frac12\|F''\|_{\infty}\|K\|_{\gamma}\right)(1+\|f+g\|_{L^1})\|f-g\|_{\shc^\gamma}.
	\end{align*}
	This concludes the proof.
\end{proof}	
From now on we  denote with $\phi$ the operator defined on $\shc^\gamma$ by 
\begin{equation}\label{def:phi}
	(\phi f)(x):=f(x)F((K*f)(x)), \ x \in \R^d.
\end{equation}

\begin{proposition}\label{pr:C_al}
  Let Assumptions \ref{ass:K} and \ref{ass:F} hold. Let $\gamma\in(0,1)$. Then $ \phi$ maps $\shc^{\gamma}\cap L^1$ into itself.
  Moreover there exists a positive constant $C_{F,K}$,
  such that,
 for all $f,g\in \shc^{\gamma}\cap L^1$
   \begin{align}\label{eq:preplin}
    \|\phi f\|_{\shc^{\gamma}}\leq C_{F,K}\|f\|_{\shc^{\gamma}}\left(1+\|f\|_{L^1}\right),
   \end{align}
\begin{equation}\label{eq:lingrow}
  \|\phi f\|_{L^1}\leq C_{F,K}\Vert f \Vert_{L^1},
 \end{equation}
  \begin{equation}\label{eq:lipgrow}
      \|\phi f-\phi g\|_{\shc^{\gamma}}\leq C_{F,K}(1+\|f\|_{\shc^{\gamma}}+\|g\|_{\shc^{\gamma}})(1+\|f+g\|_{L^1})\|f-g\|_{\shc^{\gamma}}.
    \end{equation}
  \end{proposition} 
  \begin{proof}
    We notice that the inequality \eqref{eq:preplin} is a special case of bound \eqref{eq:fg-L1} of Lemma \ref{lm:prep0} setting $g = f$. \eqref{eq:lingrow} comes  from the direct calculation
  \begin{equation}\label{eq:L1normphi}
    \|\phi f\|_{L^1}= \int_{\R^d} |f(x)F(K*f(x))|dx\leq \|F(K*f)\|_{\infty} \int_{\R^d} |f(x)|dx\leq (\vert F(0)\vert +\|F'\|_{\infty}\|K\|_{\infty})
      \Vert f \Vert_{L^1}.
  \end{equation}
  Finally we show \eqref{eq:lipgrow}. Expanding the $\shc^\gamma$-norm, according to
\eqref{eq:holdnorm},
  we get
      \begin{align*}
           \|\phi f-\phi g\|_{\shc^{\gamma}}&=\|\phi f-\phi g\|_{\infty}+\|\phi f-\phi g\|_{\gamma}\\
           &=\sup_{x\in\R^d}|\phi f(x)-\phi g(x)|+ \sup_{\underset{x\not=y}{x,y\in\R^d}}\frac{\vert\phi f(x)-\phi f(y)-\phi g(x)+\phi g(y)\vert}{|x-y|^{\gamma}} \\
           &=: \sup_{x\in\R^d}I_1(x)+ \sup_{\underset{x\not=y}{x,y\in\R^d}}\frac{I_2(x,y)}{|x-y|^{\gamma}}.
      \end{align*}
      We estimate the two summands separately. For $x \in \R^d$ we have
      \begin{align} \label{eq:I1}
      \notag  I_1(x)=&|f(x) F(K*f(x))-g(x) F(K*g(x))|\\ 
 \notag                =&|(f(x)-g(x))F(K*f(x))+g(x) (F(K*f(x))-F(K*g(x)))|\\
       \notag           \leq & \|f-g\|_{\infty} \|F(K*f)\|_{\infty}+\|g\|_{\infty} \Vert F' \Vert_\infty |K*f(x)-K*g(x)|\\
             \notag    \leq & \|f-g\|_{\infty} \left(\vert F(0) \vert+ \|F'\|_{\infty}\|K*f\|_{\infty}\right)+\|g\|_{\infty}\|F'\|_{\infty}|K*(f-g)(x)|\\
            \notag    \leq & \|f-g\|_{\infty} \left(\vert F(0) \vert+ \|F'\|_{\infty}\|K\|_{L^1}\|f\|_{\infty}\right)+\|g\|_{\infty}\|F'\|_{\infty}\|K\|_{L^1}\|f-g\|_{\infty}\\
      \notag     = & \left(\vert F(0) \vert+ \|F'\|_{\infty}(\|f\|_{\infty}+\|g\|_{\infty})\right)\|f-g\|_{\infty}\\
            \leq & C_F\left(1+ \|f\|_{\infty}+\|g\|_{\infty}\right)\|f-g\|_{\infty},
        \end{align}
        where for the last equality we used that $\|K\|_{L^1}=1$. For what concerns the second summand, for $x, y \in \R^d$, we have
        \begin{align}\label{eq:I2}
            \notag I_2(x,y)=&|f(x)F(K*f(x))-f(y)F(K*f(y))-g(x)F(K*g(x))+g(y)F(K*g(y))|\\
           \notag =&\big|\big(f(x)-g(x)-f(y)+g(y)\big)F(K*f(x))+(f(y)-g(y))\big(F(K*f(x))-F(K*f(y))\big)+\\ 
          \notag  &+\big(g(x)-g(y)\big)\big(F(K*f(y))-F(K*g(y)\big)\\
            \notag&+g(x)\big(F(K*f(x))-F(K*g(x))-F(K*f(y))+F(K*g(y))\big)\big|\\
            \notag\leq & \|f-g\|_{\gamma}|x-y|^{\gamma}\|F(K*f)\|_{\infty}+\|f-g\|_{\infty}\|F'\|_{\infty}\|K*f\|_{\gamma}|x-y|^{\gamma}+ \\
           \notag &+\|g\|_{\gamma}|x-y |^{\gamma}\|F(K*f)-F(K*g)\|_{\infty}+\|g\|_{\infty}\|F(K*f)-F(K*g)\|_{\gamma}|x-y |^{\gamma}\\
          \notag  \leq & \|f-g\|_{\gamma}|x-y|^{\gamma}  \|F(K*f)\|_{\shc^\gamma}+\|f-g\|_{\infty}\|F'\|_{\infty}
                         \|K*f\|_{\gamma}|x-y|^{\gamma}+ \\
           	\notag &+\|g\|_{\gamma}|x-y |^{\gamma}\|F(K*f)-F(K*g)\|_{\infty}+\|g\|_{\infty}\|F(K*f)-F(K*g)\|_{\gamma}|x-y |^{\gamma}\\
             \leq &C_F(1+ \|f\|_{\shc^\gamma})\|f-g\|_{\shc^\gamma}|x-y|^{\gamma}+\|g\|_{\shc^\gamma}\|F(K*f)-F(K*g)\|_{\shc^\gamma}|x-y |^{\gamma},
  \end{align}
  having used  \eqref{eq:fg-alpha} from  Lemma \ref{lm:prep0}, with $g =1$,
  in the last inequality,  together with \eqref{eq:Kstarv} for bounding $\|K*f\|_{\gamma}$. Using now \eqref{eq:FC_al} from Lemma \ref{lm:prep2} we have an upper bound for $\|F(K*f)-F(K*g)\|_{\shc^\gamma}$, so 
   dividing $I_2(x,y) $ by $|x-y|^{\gamma}$ in \eqref{eq:I2},
    where we have inserted the estimate
    \eqref{eq:FC_al}, one gets
     \begin{align} \label{eq:I2bis}
       \notag  \frac{I_2(x,y)}{|x-y|^{\gamma}}\leq &C_F(1+\|f\|_{\shc^\gamma})\|f-g\|_{\shc^\gamma}+ \|g\|_{\shc^\gamma} C_{F,K}(1+\|f+g\|_{L^1})\|f-g\|_{\shc^\gamma}\\
       \notag   \leq &\left(C_F(1+\|f\|_{\shc^\gamma}) + \|g\|_{\shc^\gamma}C_{F,K}(1+\|f+g\|_{L^1}) \right)\|f-g\|_{\shc^\gamma}\\
       \leq&C_{F,K}(1+\|f\|_{\shc^\gamma}+\|g\|_{\shc^\gamma})(1+\|f+g\|_{L^1})\|f-g\|_{\shc^\gamma}.
     \end{align}
     Summing up  \eqref{eq:I1} and \eqref{eq:I2bis},  we finally obtain \eqref{eq:lipgrow}.
 \end{proof}

 \section{The  singular non-linear
  and linearised Fokker-Planck PDEs}\label{sc:FP}

In this section we will always suppose  the validity of
Assumptions \ref{ass:beta} 
and \ref{ass:F} and  $K\in\shp(\R^d)$ to be a Borel
probability measure.
Let also $ v_0 \in \shc^{\be+}$.
We first give the definition of solution to the non-local singular non-linear Fokker Planck PDE 
 \begin{equation}\label{eq:FPlim}
    \left\{
     \begin{array}{l}
     \partial_t v=\frac12\Delta v-\text{div}(vF(K*v)b)\\
       v(0)=v_0
     \end{array}
      \right.
    \end{equation}
and to its associated linearised 
version. Then we state a fundamental a priori property that any solution of \eqref{eq:FPlim} satisfies.
    
\subsection{Definitions}

Before introducing the notion of weak and mild solutions to \eqref{eq:FPlim}
we state a preliminary lemma.

\begin{lemma}\label{lm:WPdiv_arg}
  Let $v\in C_T\shc^{\be}$ and let Assumptions \ref{ass:beta}
  and \ref{ass:F} hold. Then, the product $\phi(v)b$ is a well-defined element of $C_T\shc^{(-\be)+}$, where we recall that
  $\phi$ was defined in \eqref{def:phi}.
\end{lemma}
\begin{proof}

  Since $v \in C_T \shc^\beta$, 
  by Corollary \ref{lm:BNcont} with $g=f=v$, we have $t \mapsto \phi(v(t)) \in C_T \shc^\beta$.

Let $\alpha < \beta$ such that $b \in C_T\shc^{-\alpha}$
 Since $\beta-\alpha >0$, by
\eqref{eq:bonyt} $t \mapsto \phi(v(t))b(t) \in  C_T \shc^{-\alpha}$.
\end{proof}

\begin{definition}\label{def:WsolFP}
 We say that $v\in C_T\shc^{\beta} $ is a weak solution to \eqref{eq:FPlim}
  if for all $t \in [0,T]$ and  $\varphi\in\shs(\R^d)$ 
\begin{equation}\label{eq:WsolFP}
  \langle v(t), \varphi\rangle-\langle v_0,\varphi \rangle=\int_0^t \frac12 \langle v(s) , \Delta \varphi \rangle ds + \int_0^t  \langle  (v(s)F(K*v(s))) \cdot b(s) ,\nabla \varphi \rangle ds, 
\end{equation}
where the dual pairing is the one of $\shs,\shs'$ and $\cdot$ is the pointwise
product. Lemma
\ref{lm:WPdiv_arg} shows
that the right-hand side of \eqref{eq:WsolFP} is well-defined.
\end{definition}
\begin{definition}\label{def:MsolFP}
We say that $v\in C_T\shc^{\beta} $ is a   mild solution to \eqref{eq:FPlim} if  for all $t\in[0,T]$
\begin{equation}\label{eq:MsolFP}
  v(t)= P_{t}v_0-\int_0^t P_{t-s}\left({\rm div}
    \left((v(s)F(K*v(s)))\cdot b(s)\right)\right)ds.
\end{equation}
\end{definition}
\begin{remark}\label{rmk:lmIssoglio}
  \begin{enumerate}
   \item Since $v_0 \in C^\nu$ for some $\nu > \beta$, 
    by Schauder estimates Lemma \ref{lm:schauder}
    we have that $t \mapsto P_tv_0 \in C_T\shc^{\beta+}$, see \eqref{eq:schauder2}.
  \item
By Lemma \ref{lm:issoglio}, we remark that the integral term of \eqref{eq:MsolFP} belongs to $ C_T\shc^{(1 - \beta)-}$, i.e. for every $\alpha < 1- \beta$,
the same integral belongs to $ C_T\shc^{\alpha}$.
\item The right-hand side of \eqref{eq:MsolFP} belongs to 
  $ C_T\shc^{\beta+}$ and in particular $ C_T\shc^{\beta}$.
\item If $v_0 \in \shc^{(1 - \beta)-}$, then, a posteriori,
  the mild solution belongs to $ C_T\shc^{(1 - \beta)-}$.
\end{enumerate}
\end{remark}

\begin{lemma}\label{lm:W-M}
  Weak and mild solutions for \eqref{eq:FPlim} are equivalent.
  \end{lemma}
 \begin{proof} 
   For $v \in C_T \shc^{\beta}$ we define $h_v$
   by $h_v(s) = -{\rm div}(\phi(v(s))b(s))$, which belongs to
   $C_T \shc^{(-\beta - 1)+}$
   taking into account Lemma \ref{lm:WPdiv_arg} and  Lemma \ref{lm:bern}.

   \textit{(Mild  $\implies$ Weak).} If $v$ is a mild solution, we apply 
   Lemma 2.1 of \cite{issoglio_russoMK} with $g = g_v$.

   \textit{(Weak $\implies$ mild).} For the converse, let $v$ be a weak solution, and define $u$ such that $u(t)= P_tv_0-\int_0^t P_{t-s}h_v(s)ds, t \in[0,T]$. Again by  Lemma 2.1 of  \cite{issoglio_russoMK} with $g = h_v$, $u$ is a solution of the (linear) PDE
             \begin{equation} \label{eq:linPDE}
            \partial_t u-\frac12\Delta u+ h_v=0,\quad u(0)=v_0,
          \end{equation}
          in the sense of distributions. By definition $v$ is also a weak solution of \eqref{eq:linPDE}. Now  the difference function $w:=u-v$ is a solution of \eqref{eq:linPDE} with $h_v = 0$ and $w(0) = 0$. The fact that $w \equiv 0$ is a consequence of the uniqueness statement in Lemma 2.1 of \cite{issoglio_russoMK}. 
 \end{proof}

\subsection{Relation with the linearised singular Fokker-Planck PDE}\label{ssc:linearFPRel}
 
 If in \eqref{eq:FPlim} we set $F \equiv 1$ and $b=g\in C_T\shc^{(-\be)+}$, we end up with the linear Fokker-Planck PDE
 \begin{equation}\label{eq:FPlin}
 \left\{
 \begin{array}{l}
   \partial_t v=\frac12\Delta v-{\rm div}(v(s)g(s))\\
    v(0)=v_0.
    \end{array}
    \right.
 \end{equation}
We point out that in this case  equivalence of weak and mild solution still holds. 
We state now a significant property fulfilled by
 any solution of \eqref{eq:FPlim}.
By assumption there is $0 < \alpha < \beta$ such that $b \in C_T\shc^{-\alpha}$.
 \begin{notation} \label{not:gw}
 For $w \in C_T\shc^{\beta}$ we define
 \begin{equation} \label{eq:gw}
  g_w(t):=F(K*w(t))b(t), \   t \in [0,T].
\end{equation}
By Corollary \ref{lm:BNcont} with $g=1$  and $f=w$, we have
$F(K*w) \in C_T\shc^\beta $.
Consequently
\eqref{eq:bonyt} implies that  $ g_w \in C_T\shc^{-\alpha}$ so that
 in particular it belongs to  $C_T\shc^{(-\beta)+}$.
  \end{notation}

  \begin{theorem}
    \label{thm:lin-lim}
   Let $v \in C_T \shc^\beta$ be a
   solution of \eqref{eq:FPlim}.
   \begin{enumerate}
   \item Then  $v$
   is also a solution of \eqref{eq:FPlin} with
   $g = g_v$.
 \item Suppose moreover that $v_0 \in \shc^{\beta+}$ is a non-negative function such that $ \|v_0\|_{L^1} =1$.
   Then $v$ is non-negative and  conserves the unit mass at any time,
   that is
   $$ v(t) \ge 0,\  \int_{\R^d} v(t,x) dx = 1, \ \forall t \in [0,T].$$
   \end{enumerate}
 \end{theorem} 
 \begin{proof}
\begin{enumerate}
\item This item follows by the associativity property stated in Remark \ref{rmk:associativity}.
  \item It is a direct consequence of previous item and Corollary \ref{cor:mass-con} with $B=g_v$.
\end{enumerate}
   \end{proof}

\section{A class of singular linear Fokker-Planck PDE}\label{ssc:linearFP}

This section is devoted to the study of well-posedness and related properties of the PDE
\eqref{eq:FPlin} with $g\in C_T\shc^{(-\be)+},  v_0 \in \shc^{\be+}$.

\begin{theorem}\label{thm:exunlinFP}
  Let $g\in C_T\shc^{(-\be)+}$ and  $ v_0 \in \shc^{\be+}$.
  \begin{itemize}
  \item[(a)]

     There exists a unique
    weak solution $v \in C_T \shc^{\beta}$ of \eqref{eq:FPlin}.
  \item[(b)] The unique solution $v$ in point (a) 
    belongs to $C_T \shc^{\beta+}$. Moreover if  $v_0 \in \shc^{(1 - \beta)-}$, then
 $v$ belongs to $ C_T\shc^{(1 - \beta)-}$.
   \end{itemize}
\end{theorem}

\begin{proof}

Taking into account Lemma \ref{lm:W-M}
we can instead consider (mild) solutions of the form
  \begin{equation}\label{eq:mildlinFP}
  v(t)=P_tv_0-\int_0^t P_{t-s}{\rm div}(v(s)g(s))ds.
\end{equation}

We now prove item (a). 
\begin{description}
\item{(i)}
 Let us denote $\shi$ the map
  defined by
\begin{equation}\label{eq:I}
\shi(v)(\cdot):= P_\cdot v_0-\int_0^\cdot P_{\cdot-s}{\rm div}(v(s)g(s))ds.
\end{equation}
By Lemma \ref{lm:WPdiv_arg} and Remark \ref{rmk:lmIssoglio} 3., with $F = 1, g = b$,
$\shi$ is 
is well-defined and
maps $C_T\shc^\beta$ to itself. 
\item{(ii)}
  Now we prove the contraction property for  $v\mapsto \shi(v)$
  under the equivalent $\rho$-norm (see Definition \ref{def:rhonorm}), for
a suitable large enough $\rho$,
which will be defined later.
More particularly we will show the existence of $\rho > 0$ depending on
$\alpha, \beta, \Vert g \Vert_{C_T \shc^{-\alpha}}$ 
  such that,
for any  $v_1,v_2\in C_T\shc^{\beta}$,
 \begin{equation} \label{eq:contraction}
    \|\shi(v_1)-\shi(v_2)\|^{(\rho)}_{C_T\shc^{\beta}} \le
    \frac{1}{2}  \|v_1- v_2\|^{(\rho)}_{C_T\shc^{\beta}}.
  \end{equation}
  Let $\alpha < \beta$ such that $g \in C_T \shc^{-\alpha}$ and 
  let us set  $\theta:= \frac{\beta + \alpha + 1}{2}.$
We obtain the following chain of inequalities, where at the fourth line we have
used \eqref{eq:schauder1} with $\gamma = -\alpha - 1$ so that $\beta = 2 \theta + \gamma$
and at the fifth line \eqref{eq:bonyt} and Lemma \ref{lm:bern}:
  \begin{align*}
    \|\shi(v_1)-\shi(v_2)\|^{(\rho)}_{C_T\shc^{\beta}}=&\left\|\int_0^\cdot
    P_{\cdot-s}{\rm div}(g(s)(v_1(s)-v_2(s)))ds\right\|^{(\rho)}_{C_T\shc^{\beta}}\\
   = & \sup_{t\in[0,T]} \left\{e^{-\rho t} \left\|\int_0^tP_{t-s}{\rm div}(g(s)(v_1(s)-v_2(s)))ds\right\|_{\shc^{\beta}}\right\}\\
     \leq&   \sup_{t\in[0,T]} \left\{e^{-\rho t} \int_0^t\left\|P_{t-s}{\rm div}(g(s)(v_1(s)-v_2(s)))\right\|_{\shc^{\beta}}ds\right\}\\
    \leq&   \sup_{t\in[0,T]} \left\{ce^{-\rho t}\int_0^t(t-s)^{-\theta}
     \|{\rm div}(g(s)(v_1(s)-v_2(s)))\|_{\shc^{-(\alpha+1)}}ds\right\}\\
    \leq&   \sup_{t\in[0,T]} \left\{c\int_0^t(t-s)^{-\theta}e^{-\rho (t-s)}\left\|g\|_{C_T\shc^{-\alpha}}
          e^{-\rho s}\|v_1(s)-v_2(s)\right\|_{\shc^{\beta}}ds\right\}\\
  \leq&   \|g\|_{C_T\shc^{-\alpha}}\|v_1-v_2\|^{(\rho)}_{C_T\shc^{\beta}}c \sup_{t\in[0,T]} \left\{\int_0^t(t-s)^{-\theta}e^{-\rho (t-s)}ds\right\}\\
  \leq&   \|g\|_{C_T\shc^{-\alpha}}\|v_1-v_2\|^{(\rho)}_{C_T\shc^{\beta}}c \int_0^\infty s^{-\theta}e^{-\rho s}ds\\
 =& \|g\|_{C_T\shc^{-\alpha}}\|v_1-v_2\|^{(\rho)}_{C_T\shc^{\beta}}c\rho^{\theta-1}\Gamma(1-\theta).
 \end{align*}
 Since $\theta<1$ by the assumption made on $\al$ and $\be$, the map $\rho\mapsto \rho^{\theta-1}$
converges to zero when $\rho$ goes to $+\infty$.
Thus, choosing $\rho$ so that
$$ \|g\|_{C_T\shc^{-\alpha}} c\rho^{\theta-1}\Gamma(1-\theta) = \frac{1}{2},$$
we obtain  \eqref{eq:contraction}.
\item{(iii)} By Banach fixed point theorem the proof of item (a) is concluded.
  \end{description}
  Item (b) follows by Remark \ref{rmk:lmIssoglio} 4. with $F= 1, g = b$.
  \end{proof}

We continue discussing the dependence on the coefficients of the solution $v$ to \eqref{eq:FPlin}.
\begin{theorem}\label{thm:exunlinFPCont}
  Let $g$  (resp. $g_n, n \in \N$) be an element of   $C_T\shc^{(-\be)+}$   and let $v_0$ (resp. $v_0^n$) in $\shc^{\beta+}$.
  Let $v$ (resp. $v_n$)
  be the solution
to 
\eqref{eq:FPlin} (resp. with $g=g_n$ and $v_0 = v_0^n$), 
provided by Theorem \ref{thm:exunlinFP}.
If $(g_n)$ is a sequence converging to $g$ in $C_T\shc^{(-\be)+}$ and $(v_0^n)$
converges to $v_0$ in $\shc^{\beta+}$, then $v_n\to v$ in $C_T\shc^{\beta}$.
\end{theorem}
\begin{proof}
We emphasize that
 the sequence $(v_n)$ belongs to $ C_T\shc^{\beta}$.
 By assumption, there is $\nu >\beta$ such that  $v^n_0$ converges to $v_0$ in $\shc^{\nu}$.
 Let $ \alpha < \beta$ such that $(g_n)$ converges to $g$ in $C_T\shc^{-\alpha}$.
  We want to show that  $v_n$ converges to $v$ in $C_T\shc^{\beta}$ as $n \rightarrow +\infty$.
We define  $\theta:=\frac{\al+\be+1}2$.  
  Using Schauder's and Bernstein inequalities Lemmata \ref{lm:schauder}-\ref{lm:bern}, for every $t \in [0,T]$, we have
\begin{align} \label{eq:sben}
  \|v(t)-v_n(t)\|_{\shc^\beta}\leq&\|P_t(v_0-v_0^n)\|_{\shc^\beta}+\left\|\int_0^tP_{t-s}{\rm div}(v(s)g(s)-v_n(s)g_n(s))ds\right\|_{\shc^\beta}
  \notag \\
\leq&c\|(v_0-v_0^n)\|_{\shc^\nu}+\int_0^t\left\|P_{t-s}{\rm div}(v(s)g(s)-v_n(s)g_n(s))\right\|_{\shc^\beta}ds \notag\\
\leq&c\|(v_0-v_0^n)\|_{\shc^\nu}+c\int_0^t(t-s)^{-\theta}\left\|{\rm div}(v(s)g(s)-v_n(s)g_n(s))\right\|_{\shc^{-\alpha-1}}ds \notag \\
  \leq&c\|(v_0-v_0^n)\|_{\shc^\nu}+c\int_0^t(t-s)^{-\theta}\left\|v(s)g(s)-v_n(s)g_n(s)\right\|_{\shc^{-\alpha}}ds  \\
   \leq&c\|(v_0-v_0^n)\|_{\shc^\nu}+c\int_0^t(t-s)^{-\theta}\left\|v(s)(g(s)-g_n(s))-(v_n(s)-v(s))g_n(s)\right\|_{\shc^{-\alpha}}ds \notag   \\
\leq&c\|(v_0-v_0^n)\|_{\shc^\nu}+c\int_0^t(t-s)^{-\theta}\left[\|v(s)(g(s)-g_n(s))\|_{\shc^{-\alpha}}+\|(v_n(s)-v(s))g_n(s)\|_{\shc^{-\alpha}}\right]ds \notag\\
  \leq&c\|(v_0-v_0^n)\|_{\shc^\nu}+c\int_0^t(t-s)^{-\theta}\left[\|v(s)\|_{\shc^{\beta}}\|g(s)-g_n(s)\|_{\shc^{-\alpha}}+\|v_n(s)-v(s)\|_{\shc^\beta}\|g_n(s)\|_{\shc^{-\alpha}}\right]ds,
        \notag
\end{align}
having used  \eqref{eq:bony} in the last bound. 
Since $g_n\to g$ in $C_T\shc^{-\alpha}$ it is true that $\sup_n\|g_n\|_{C_T\shc^{-\alpha}}<\infty$. Thus, \eqref{eq:sben} implies
\begin{align*}
  \|v(t)-v_n(t)\|_{\shc^\beta}\leq& c\|(v_0-v_0^n)\|_{\shc^\nu}+c \|v\|_{C_T\shc^\beta}\|g-g_n\|_{C_T\shc^{-\alpha}}\int_0^t(t-s)^{-\theta}ds\\
&+c\int_0^t(t-s)^{-\theta}\|v_n(s)-v(s)\|_{\shc^\beta}\sup_n\|g_n\|_{C_T\shc^{-\alpha}}ds\\
\leq&   c\|(v_0-v_0^n)\|_{\shc^\nu}+c\|v\|_{C_T\shc^\beta}\|g-g_n\|_{C_T\shc^{-\alpha}}\frac{T^{1-\theta}}{1-\theta}\\
&+c\int_0^t(t-s)^{-\theta}\|v_n(s)-v(s)\|_{\shc^\beta}\sup_n\|g_n\|_{C_T\shc^{-\alpha}}ds.
\end{align*}
We use the generalisation of Gronwall's Lemma  stated in \cite{ye} and taking the supremum for $t \in [0,T]$ in previous inequality, we get
\begin{equation} \label{eq:genGron} 
\|v-v_n\|_{C_T\shc^\beta}\leq c\left[\|v_0-v_0^n\|_{\shc^\nu}+\|v\|_{C_T\shc^\beta}\|g-g_n\|_{C_T\shc^{-\alpha}}\frac{T^{1-\theta}}{1-\theta}\right]E_{1-\theta}\left(c\sup_n\|g_n\|_{C_T\shc^{-\alpha}}T^{1-\theta}\Gamma(1-\theta)\right),
\end{equation}
where $\Gamma(x)$ is the Euler Gamma function and
\begin{equation} \label{eq:Mittag}
  E_{1-\theta}(x):=\sum_{k=0}^\infty\frac{x^k}{\Gamma((1-\theta) k+1)}
\end{equation}
  is the Mittag-Leffler function of parameter $1 - \theta$.
This concludes the proof.
\end{proof}

At this point we want to show that, whenever the initial condition
$\eta_0$ is non-negative with unitary mass (and smooth enough), the solution $\eta$ 
to the linear Fokker-Planck PDE \eqref{eq:FPlin},
provided
by Theorem \ref{thm:exunlinFP}, 
preserves previous properties at all times. 
This  relies on the following key lemma. 

\begin{lemma}\label{lm:vinL1}
  Let $h$ be a vector-valued function
  in $C_TC_b^{2}(\R^d)$, $\eta_0\in L^1\cap
C_b^3(\R^d)$ be a non-negative function. 
 Let $\eta\in C_T\shc^{\beta}$ be the unique solution in the sense of  Definition   \ref{def:WsolFP}, 
ensured by Theorem~\ref{thm:exunlinFP}, 
to 
\begin{equation}\label{eq:FPlinProb}
   \partial_t \eta =\frac12\Delta \eta -{\rm div}(\eta(s)h(s)),\ \eta(0)=\eta_0.
   \end{equation}
   Then 
    $\eta(t) \ge 0$
   and  $\|\eta(t)\|_{L^1}=
   \|\eta_0\|_{L^1}$,
   for all $t\in[0,T]$.
 \end{lemma}

\begin{proof}
 The linear Fokker Planck PDE  \eqref{eq:FPlinProb} can be rewritten as
    \begin{equation}\label{eq:FPlunardi}
    \partial_t \eta=\frac12\Delta \eta - \nabla \eta h- \eta{\rm div} h ,\ \eta(0)=\eta_0,
  \end{equation}
  which admits a classical solution in $ C_b^{1,2+\epsilon}([0,T]\times \R^d)$, for some $\epsilon  > 0$, by
  Theorem 5.1.9 of \cite{lunardi95}. This is therefore also a solution in the sense of distributions. Now, following the reasoning in Remark 4.12 in \cite{issoglio_russoPDEa}, we know that $C_b^{1,2+\epsilon}([0,T]\times \R^d)
  \subseteq C_T\shc^{(1+\be)+}(\R^d) \subseteq C_T\shc^{\be}(\R^d)   $. Thus, previous classical solution
   is a weak solution in the sense of Definition \ref{def:WsolFP}
   with $F \equiv 1$, $b = h$, $v_0 = \eta_0$. By item (a)  of Theorem \ref{thm:exunlinFP},
   which states in particular uniqueness for weak solutions of \eqref{eq:FPlunardi},
 we can conclude that the unique weak solution $\eta$ to \eqref{eq:FPlinProb}  belongs to $C_b^{1,2+\epsilon}([0,T]\times\R^d)$.
    
 Now we denote by $u$ the time reversal on the interval $[0,T]$  of $\eta$ defined by
 \begin{equation}\label{eq:FPrever}
 u(t,x):=\eta(T-t,x),\ \forall(t,x)\in[0,T]\times\R^d,
 \end{equation}
 which is the classical solution of the Kolmogorov PDE 
 \begin{equation}\label{eq:Kolm}
 \partial_t u +\frac12 \Delta u  - \nabla u \hat h - u{\rm div} {\hat h} =0,\ u(T)=\eta_0,
\end{equation}
where ${\hat h}(t) = h(T-t), t \in [0,T].$
 Let $(X_t^{s,x})$ be the stochastic flow solution of 
 \begin{equation}\label{eq:stochflow}
 dY_t=dW_t-\hat h(t,Y_t)dt,\ Y_s=x,
 \end{equation}
 with $s\leq t$.
 Then by Feynman-Kac formula we obtain a representation of $u$ in terms of the stochastic flow $(X_t^{s,x})$, that is
 \begin{equation}\label{eq:FKun}
   u(s,x)=\E\left[\eta_0(X_T^{s,x})\exp{\int_s^T (-{\rm div}{\hat h})(r,X_r^{s,x})}dr\right],
 \end{equation}
which implies that  $\eta$ has representation
 \begin{equation}\label{eq:FKvn}
 \eta(s,x)=\E\left[\eta_0(X_T^{T-s,x})\exp{\int_{T-s}^T (-{\rm div} {\hat h})(r,X_r^{T-s,x})}dr\right].
 \end{equation}
This implies that $\eta(t,x)\geq 0$ for all $(t,x)\in[0,T]\times \R^d$. Now we consider the mild formulation
 for \eqref{eq:FPlinProb}.
 We define the mapping $J$ on  $B_TL^1$  by
\begin{equation}\label{def:tau}
  \zeta  \mapsto  J_\cdot(\zeta):=P_\cdot\eta_0-\displaystyle\int_0^\cdot P_{\cdot-s}{\rm div}(\zeta(s)h(s))ds,
  \end{equation}
 where we recall that $B_T L^1$ is  endowed with the norm $\Vert \eta \Vert_{B_T L^1} : = \sup_{t \in [0,T]} \Vert \eta(t) \Vert_{L^1}$.
  We want to show that $J$ takes values in  $B_TL^1$
  and it is a contraction with respect to an equivalent norm.
   Concerning the stability of $J$,
for $ \zeta \in  B_TL^1 $, $ t \in [0,T]$ we have
\begin{equation}\label{eq:Jt}
\| J_t(\zeta)\|_{L^1}\leq\|P_t\eta_0\|_{L^1}+\left\|\int_0^tP_{t-s}{\rm div}(\zeta(s)h(s))ds\right\|_{L^1}= \|\eta_0\| _{L^1}+\left\|\int_0^tP_{t-s}{\rm div}(\zeta(s)h(s))ds\right\|_{L^1},
\end{equation}
since the heat operator $P_t$ conserves the $L^1$-norm. Thus we are left to bound the second summand, which gives
\begin{align}\label{eq:stab}
\notag
  \left\|\int_0^tP_{t-s}{\rm div}(\zeta(s)h(s))ds\right\|_{L^1}&=\left\|\int_0^t{\rm div} (P_{t-s}(\zeta(s)h(s)))ds\right\|_{L^1}
  \\ \notag
&=\int_{\R^d}\left|\int_0^t\int_{\R^d} \nabla_x p_{t-s}(x-y)\zeta(s,y)h(s,y)dyds\right|dx\\
&\leq   \int_{\R^d}\int_0^t\int_{\R^d}\left| \nabla_x p_{t-s}(x-y)\zeta(s,y)h(s,y)\right|dydsdx.
\end{align}
Since there exists  $\tilde C>0$ such that
$ ae^{-a^2}\leq \tilde  C e^{\frac{-a^2}2},\ \forall a\in\R$, setting $a:=\frac{|x|}{\sqrt t}$ we have for all $t>0$
 \begin{equation*}
 |\nabla p_t(x)|=\frac{|x|}{t}p_t(x)=\frac{|x|}{t}e^{-\frac{|x|^2}{2t}}\frac1{(2\pi t)^{d/2}}\leq  \frac{\tilde C}{\sqrt t} e^{-\frac{|x|^2}{4t}}\frac{2^{d/2}}{(4\pi t)^{d/2}}= \frac{C}{\sqrt{t}}p_{2t}(x),
  \end{equation*}
  where $C:=\tilde C 2^{d/2}$. 
  So using this in \eqref{eq:stab} 
 we get 
  \begin{align}\label{eq:stab2}
  \notag
   \left\|\int_0^tP_{t-s}{\rm div}(\zeta(s)h(s))ds\right\|_{L^1}& \leq
  \int_{\R^d}\int_0^t\int_{\R^d}\frac{C}{\sqrt{t-s}}\left|p_{2(t-s)}(x-y)\zeta(s,y)h(s,y)\right|dydsdx\\
\notag
&\leq  \int_0^t\int_{\R^d}\int_{\R^d}\frac{C}{\sqrt{t-s}}\left|p_{2(t-s)}(x-y)\zeta(s,y)h(s,y)\right|dxdyds\\
&=   \int_0^t\int_{\R^d}\frac{C}{\sqrt{t-s}}\left|\zeta(s,y)h(s,y)\right|dyds,
\end{align}
since $\displaystyle{\int_{\R^d}p_{2(t-s)}(x-y)dx=1}$. 
Thus, using the fact that $h$ is bounded in \eqref{eq:stab2}, bound  \eqref{eq:Jt} gives
\begin{align*}
  \Vert J_t(\zeta)\Vert_{ L^1}    \le \Vert \eta_0\Vert_{L^1} +  \|h\|_{\infty} \int_0^t \frac{C}{\sqrt{t-s}}\|\zeta(s,\cdot)\|_{L^1}ds
     \le \Vert \eta_0\Vert_{L^1} +  \|h\|_{\infty}C\sqrt{T}\|\zeta\|_{B_T L^1}
\end{align*}
and  taking the supremum over $t\in[0,T]$ the stability is proved.

Now we focus on the contraction property of the operator $J$. Let us consider $\zeta_1,\zeta_2 \in B_TL^1$.
Coming back to \eqref{def:tau},
using \eqref{eq:stab2} with $\zeta = \zeta_1 - \zeta_2$ and boundedness of $h$
we easily get
\begin{equation} \label{eq:J12}
  \|J_t(\zeta_1)-J_t(\zeta_2)\|_{L^1}
   =
   \| \int_0^tP_{t-s}{\rm div}((\zeta_1(s)-\zeta_2(s))h(s))ds \|_{L^1}
   \leq
  \|h\|_{\infty} \int_0^t \frac{C}{\sqrt{t-s}} \left\| \zeta_1(s)-\zeta_2(s)\right\|_{L^1}ds.
\end{equation}
For $\rho>0$ let us define the norm $\|\cdot\|^{(\rho)}_{B_TL^1}:=\sup_{t\in[0,T]}e^{-\rho t}\|\cdot\|_{B_TL^1}$, which is equivalent to the $\|\cdot\|_{B_TL^1}$ norm. 
We multiply both sides of \eqref{eq:J12} by $e^{-\rho t}$ to get
\begin{align*}
  e^{-\rho t}\|J_t(\zeta_1)-J_t(\zeta_2)\|_{L^1}\leq&    \|h\|_{\infty} \int_0^t \frac{Ce^{-\rho(t-s)}}{\sqrt{t-s}} e^{-\rho s}\left\| \zeta_1(s)-\zeta_2(s)\right\|_{L^1}ds\\
                                                    &\leq C  \|h\|_{\infty}  \left\| \zeta_1-\zeta_2\right\|^{(\rho)}_{B_TL^1} \int_0^t \frac{e^{-\rho s}}{s^{\frac{1}{2}}} ds \\
  &\leq    C  \rho^{-\frac{1}{2}}  \|h\|_{\infty}  \left\| \zeta_1-\zeta_2\right\|^{(\rho)}_{B_TL^1}   \int_0^\infty \frac{e^{-s}}{s^{\frac{1}{2}}} ds  \\
  &\leq   C \|h\|_{\infty} \rho^{-\frac12}\sqrt{\pi}\left\| \zeta_1-\zeta_2\right\|^{(\rho)}_{B_TL^1},
\end{align*}
since $ \int_0^\infty \frac{e^{-s}}{s^{\frac{1}{2}}} ds = \sqrt{\pi}$.
Therefore, taking the supremum over $[0,T]$ leads us to the inequality
\begin{equation*}
\|J(\zeta_1)-J(\zeta_2)\|^{(\rho)}_{B_TL^1}\leq C \|h\|_{\infty}  \rho^{-\frac12}\sqrt{\pi}\left\| \zeta_1-\zeta_2\right\|^{(\rho)}_{B_TL^1}.
\end{equation*}
Since $\rho\mapsto \rho^{-\frac12}$ is a decreasing function we can choose $\rho$ such that $C \|h\|_{\infty} \rho^{-\frac12}\sqrt{\pi}=\frac12$, thus we get
\begin{equation}
\|J(\zeta_1)-J(\zeta_2)\|^{(\rho)}_{B_TL^1}\leq\frac12\left\| \zeta_1-\zeta_2\right\|^{(\rho)}_{B_TL^1},
\end{equation}
that is, $J$ is a contraction and admits a unique fixed point $\zeta \in B_TL^1$, by Banach fixed point theorem.
%

We now show that  $\eta$ coincides with this fixed point.
The solution $\eta$ belongs to $C_T \shc^\beta$ by Theorem \ref{thm:exunlinFP} and it
is the unique fixed point of the solution map $\shi:  C_T\shc^\beta \to C_T\shc^\beta $, where $\shi$  was defined in  \eqref{eq:I} in the proof of  Theorem \ref{thm:exunlinFP}.
We define $\zeta^0(t) = P_t \eta_0, t \in [0,T]$. It belongs to  $C_T \shc^\beta \cap B_T L^1$ by assumption on $\eta_0$, \eqref{eq:schauder2}, and the conservation of the
$L^1$-norm of the heat semigroup. Therefore, the Picard iterations of $J$ and $\shi$ coincide, that is
$\zeta^{n+1}=J(\zeta^n) = \shi(\zeta^n)$ for $n \ge 0$, thus they both belong to  $C_T \shc^\beta \cap B_T L^1$. By convergence of the Picard iterations in  $C_T \shc^\beta \cap B_T L^1$, we have that
$\eta = \zeta$ so that $t \mapsto \Vert \eta(t) \Vert_{L^1}$ is bounded on $[0,T]$.  
We can now apply Theorem 2.6 of \cite{figalli} to $\eta(t,x)dx$ since the latter is a positive finite measure, solution to \eqref{eq:FPlinProb},
and retrieve that
$\|\eta(t)\|_{L^1}=  \|\eta_0\|_{L^1}, $ for all $t\in[0,T]$.
\end{proof}

The next result is an application of Lemma \ref{lm:vinL1} to the
singular linear Fokker-Planck \eqref{eq:FPlin}.

\begin{proposition}\label{lin_sing_PDE}
  Let $g\in C_T\shc^{(-\be)+}$ and  $v $ be  a solution to \eqref{eq:FPlin} in $C_T\shc^{\beta}$.
  Let $v_0 := v(0)$ belong to
  $ L^1(\R^d)\cap \shc^{\be+}$
with $\|v_0\|_{L^1}\le 1$ and  non-negative. Then $v(t) \in L^1(\R^d)$, it is non-negative
and $\|v(t)\|_{L^1}\leq 1$, for all $t\in[0,T]$.
    \end{proposition}
    \begin{proof}
By Lemma \ref{lm:Xncont} 
we can  consider a sequence $(g_n)\subseteq C_TC_b^{2}$, converging to $g$ in $C_T\shc^{(-\be)+}$ as $n \to \infty$.
We also consider the sequence
$(v^n_0)\subseteq C_b^{3}$ defined by $v^n_0:=p_{\frac1n}\ast v_0$.
Since $v_0 \in \shc^{\beta+}$, by Lemma \ref{lm:schauder} we get that $v_0^n \to v_0$ in $\shc^{\beta+}$.
By Theorem \ref{thm:exunlinFP}, for each $n$, there is
$v_n \in C_T \shc^{\beta}$,  solution to (\ref{eq:FPlin}) with $g = g_n$ and initial value $v_0 = v_0^n$.
We apply Lemma \ref{lm:vinL1} to $\eta = v_n$ where $\eta_0=v_0^n$ and $h=g_n$, to obtain that $\|v_n(t)\|_{L^1}\le 1$
and $v_n(t,x)\geq0$ for all  $(t,x)\in[0,T]\times \R^d$ and for all $n\in\N$. 
In particular the sequence $(v_n)$ lives in $C_T \shc^\beta \cap S_{L^1}$.
Since $g_n\to g$ in $C_T\shc^{(-\be)+}$ and $v_0^n\to v_0$ in $\shc^{\beta+}$, Theorem \ref{thm:exunlinFPCont},
 which states the continuity with respect to the coefficient and initial data,
 guarantees that $v_n\to v$ in $C_T\shc^{\beta}$.  By Lemma \ref{lm:B1close},  $v$ also  belongs to $ C_T\shc^{\beta} \cap S_{L^1}$
 and the conclusion follows.
\end{proof}
\begin{corollary}\label{cor:lin_sing_PDE}
  Let $g\in C_T\shc^{(-\be)+}$ and  $v $ be  a solution to \eqref{eq:FPlin} in $C_T\shc^{\beta}$
such that $v_0 := v(0)$ belongs to 
  $ L^1(\R^d)\cap \shc^{\be+}$
  and  non-negative. Then $v(t)$
 is non-negative
  for all $t\in[0,T]$ and $v \in B_TL^1$.
    
\end{corollary}
\begin{proof}
  Let $M > \|v_0\|_{L^1}$. We apply Proposition \ref{lin_sing_PDE} with
$\tilde v = \frac{v}{M}$ and we conclude.
\end{proof}


\section{Well-posedness for
  the non-local singular non-linear Fokker-Planck PDE}
\label{sc:wpnonlinear}

In this section we study well-posedness of the non-local singular non-linear
Fokker-Planck PDE \eqref{eq:FPlim}, and we divide the study of uniqueness and
existence in two subsections.
Throughout the section we assume the validity of Assumption \ref{ass:beta},
\ref{ass:K}
and \ref{ass:F}. We still suppose that $v_0\in\shc^{\be+} \cap L^1$
is a non-negative function.
By Assumption  \ref{ass:beta},  $b \in C_T\shc^{-\alpha}$ for some $\alpha < \beta$,
fixed for the Sections  \ref{ssc:nonlinU}   and \ref{ssc:nonlinear}.

We recall that,
if $w \in C_T\shc^{\beta}$
we denote by $g_w$ the element $C_T\shc^{-\alpha}$ defined, in Notation \ref{not:gw},
by
$g_w(t)=F(K*w(t))b(t), \  t \in [0,T]$.

\subsection{Uniqueness for the non-linear Fokker-Planck PDE}
\label{ssc:nonlinU}

We state and prove a  uniqueness result which is based on an a priori 
 $L^1$-bound on a given
solution to \eqref{eq:FPlim}.

\begin{proposition}\label{pr:uniFPlim}
  Let Assumptions \ref{ass:beta}, \ref{ass:K} and \ref{ass:F} hold and let
  $v_0\in\shc^{\be+} \cap L^1(\R^d)$ be a non-negative function.
  Then there exists at most one solution 
 to the non-local singular non-linear Fokker-Planck PDE \eqref{eq:FPlim}.
\end{proposition}

\begin{proof}
  Suppose there are two solutions $v$ and $y$ to \eqref{eq:FPlim}. Then $v$ and $y$ are solutions to the linearised Fokker-Planck PDE \eqref{eq:FPlin} with respectively $g=g_v$ and $g=g_y$, see Theorem \ref{thm:lin-lim}. We recall that $g_v, g_y \in  C_T\shc^{-\alpha}$, where $\alpha$
  was introduced at the beginning of Section \ref{sc:wpnonlinear}. 
  Thus by Corollary \ref{cor:lin_sing_PDE} we have that both $v,y\in B_TL^1$ and they are both non-negative.  
Let us consider the real positive number $\theta:=\frac{\al+\be+1}{2}$. Now using \eqref{eq:sben} with $v_n=y,\ g=g_v,\ g_n=g_y$ and $v_0^n = v_0$ together with \eqref{eq:bony} and \eqref{eq:bonyt} and Lemma \ref{lm:prep}, for $t \in [0,T]$, one gets
\begin{align}\label{eq:first-est1}
	\notag\|v(t)-y(t)\|_{\shc^\be}\leq& c\int_0^t(t-s)^{-\theta}\left[\|v(s)\|_{\shc^{\be}}\|g_v(s)-g_y(s)\|_{\shc^{-\al}}+\|y(s)-v(s)\|_{\shc^\be}\|g_y(s)\|_{\shc^{-\al}}\right]ds\\
	 \leq & c\int_0^t (t-s)^{-\theta}\Big[\|v(s)\|_{\shc^\be}\|F(K*v(s))-F(K*y(s))\|_{\shc^{\be}}\|b(s)\|_{\shc^{-\al}}\\
	\notag &+\|y(s)-v(s)\|_{\shc^\be}\|F(K*y(s))\|_{\shc^{\be}}\|b(s)\|_{\shc^{-\al}}\Big]ds.
\end{align}
We now control the two terms involving $K$. Using bound \eqref{eq:fg-L1} of Lemma \ref{lm:prep0} with $\g=\be$, $g=1$ and $f=y(s)$ with any fixed $s\in[0,T]$, 
we get
\begin{equation}\label{eq:second-est1}
	\|F(K*y(s))\|_{\shc^{\be}} \le C_{F,K} (1 + \Vert y(s) \Vert_{L^1}) \le C_{F,K} (1 + \Vert y \Vert_{B_TL^1}).
\end{equation}
Using \eqref{eq:FC_al} in Lemma \ref{lm:prep2} with $\g=\be$, $f=v(s)$ and $g=y(s)$,  we have
\begin{equation} \label{eq:third-est1}
	\|F(K*v(s))-F(K*y(s))\|_{\shc^{\be}}  \le    C_{F,K} (1 + \Vert v(s) + y(s) \Vert_{L^1}) \Vert v(s) - y(s) \Vert_{ \shc^\be}, \ s\in[0,T].
\end{equation}
Let us take $M = \|v\|_{C_T\shc^\be}$ and define $L:=\max\{\|v\|_{B_TL^1},\|y\|_{B_TL^1}\}$. Inserting \eqref{eq:second-est1} and \eqref{eq:third-est1} into \eqref{eq:first-est1} and rearranging the term involving $\|v(s)-y(s)\|_{\shc^\be}$, for $t \in [0,T]$, we have
\begin{align*}
	\notag\|v(t)-y(t)\|_{\shc^\be}\leq & c\int_0^t (t-s)^{-\theta}\Big[\|v(s)\|_{\shc^\be}   C_{F,K} (1 + \Vert v(s) + y(s) \Vert_{L^1}) \Vert v(s) - y(s) \Vert_{ \shc^\be}\|b(s)\|_{\shc^{-\al}}\\
	\notag &+\|y(s)-v(s)\|_{\shc^\be}C_{F,K} (1 + \Vert y \Vert_{B_TL^1})\|b(s)\|_{\shc^{-\al}}\Big]ds\\
	\leq & c\int_0^t (t-s)^{-\theta}\Big[  M C_{F,K} (1 + 2L) \Vert v(s) - y(s) \Vert_{ \shc^\be}\|b(s)\|_{\shc^{-\al}}\\
	\notag &+\|y(s)-v(s)\|_{\shc^\be}C_{F,K} (1 + L)\|b(s)\|_{\shc^{-\al}}\Big]ds\\
\leq & c\Big[M  (1 + 2L) + (1 + L)\Big]C_{F,K}\|b\|_{C_T\shc^{-\al}}\int_0^t (t-s)^{-\theta}\|y(s)-v(s)\|_{\shc^\be}ds.
\end{align*}
Using now the  generalised Gronwall's inequality \cite{ye}, we get that $v=y$, hence uniqueness.

\end{proof}
\subsection{Existence of the non-linear Fokker-Planck PDE}
\label{ssc:nonlinear}
   


 \begin{remark} \label{rmk:unit_ball}
\begin{enumerate}
\item   We recall that
\begin{equation}\label{unit_ball}
  C_T\shc^{\beta}\cap S_{L^1}=\bigg\{w\in C_T\shc^{\beta}\bigg| \forall t\in[0,T], w (t) \ge 0, \  \int_{\R^d} w(t,x)dx\leq 1\bigg\}.
\end{equation}
\item
  In order to prove existence and uniqueness for the PDE
    (\ref{eq:FPlim})
  through a fixed point argument,
we define now the  ``linearised solution  map''
\begin{equation} \label{eq:linsolmap}
  \tau:  C_T\shc^{\beta}\cap S_{L^1} \rightarrow  C_T\shc^{\beta}\cap S_{L^1},
\end{equation}
as follows.
Given $w \in C_T\shc^{\beta}\cap S_{L^1}$, we set $\tau(w):=v,$
 where  $v \in  C_T\shc^{\beta}  $ is the solution
 of \eqref{eq:FPlin},
 with $g = g_w$, whose existence and uniqueness is provided by Theorem \ref{thm:exunlinFP} 
and it belongs to $ C_T\shc^{\beta}\cap S_{L^1}$
by
Proposition \ref{lin_sing_PDE}.
\item
By  Lemma \ref{lm:B1close} we know that $ C_T\shc^{\beta}\cap S_{L^1}$ is a closed set in $C_T\shc^{\beta}$,
so that \eqref{unit_ball} is an $F$-space  under the uniform topology inherited from $C_T\shc^\beta$.
Therefore $\tau$ maps an $F$-space into itself. 
\end{enumerate}
\end{remark}
We now prove that $\tau$ maps all $w\in  C_T\shc^{\beta}\cap S_{L^1}$ into a ball of $ C_T\shc^{\beta}\cap S_{L^1}$ under the uniform topology inherited from $C_T\shc^\beta$.
In the sequel of the section, for any $M > 0$,  we will make use of the notation
\begin{equation}\label{eq:BMtau}
  B^M:=\left\{w\in  C_T\shc^{\beta}\cap S_{L^1}\vert \|w\|_{C_T\shc^\beta}\leq M\right\}.
\end{equation}
\begin{remark} \label{rmk:Boundtau}
  \begin{enumerate}
    \item For every $M > 0$, the ball $B^M$ defined in \eqref{eq:BMtau}
is a closed set of $C_T\shc^{\beta}.$
Indeed,  let $(u_n)$ in $B^M$ converging to some $u \in  C_T\shc^{\beta}$.
Taking into account Lemma \ref{lm:B1close}, it remains to show that
 $\Vert u \Vert_{C_T \shc^\beta} \le M$. This
easily follows  by Fatou's lemma
applied to the uniform norm of $u$ and
with respect to seminorm $\Vert u \Vert_{\beta}$ defined in
\eqref{eq:holdnorm}.
\item In particular $B^M$ is an $F$-space.
\end{enumerate}

  \end{remark}

 \begin{lemma}\label{lm:boundtau}
  Let Assumptions \ref{ass:beta}, \ref{ass:K} and \ref{ass:F} hold.
   Let $v_0 \in \shc^{\beta+}  \cap L^1$, non-negative, such that $\Vert v_0 \Vert_{L^1} \le 1$.
Let also $0 < \alpha < \beta$ such that $b \in C_T\shc^{-\alpha}.$ 
   Then there exists an increasing  function $M_0: \R_+ \rightarrow \R_+$
   depending on $\al,\be,C_{F,K},
   \|v_0\|_\beta$,
   such that,  $\tau$ defined in \eqref{eq:linsolmap} maps $ C_T\shc^{\beta}\cap S_{L^1}$ into $ B^{M_0}$,
   with $M_0 = M_0(\|b\|_{C_T\shc^{-\alpha}})$.
   \end{lemma}

  \begin{proof}
    We denote    $\theta:=\frac{\al+\be+1}{2}$.
Let $w \in C_T\shc^{\beta} \cap S_{L^1}$. 
 By Remark \ref{rmk:unit_ball}, point 2.\ we know that
$\tau(w) \in  C_T\shc^{\beta} \cap S_{L^1}$.
It remains to show  that $\|\tau(w)\|_{C_T\shc^\beta} \le
M_0(\|b\|_{C_T\shc^{-\alpha}})$, where
 $M_0$ has to be determined. We set $v:=\tau(w)$.
Taking $v_0^n=g_n=0$ and $g=g_w$ in \eqref{eq:sben}, which is in the proof of
Theorem \ref{thm:exunlinFPCont}, leads to
\begin{equation*}
  \|v(t)\|_{\shc^\beta}\leq c \|v_0\|_{\shc^\beta}+c\|g_w\|_{C_T\shc^{-\alpha}}\int_0^t(t-s)^{-\theta}\|v(s)\|_{\shc^{\beta}}ds,  \  t \in [0,T].
 \end{equation*}
 So, applying again the generalised Gronwall's inequality in \cite{ye}, we obtain
\begin{equation}\label{eq:boundtau}
  \|\tau(w)\|_{C_T\shc^\beta}\leq c\|v_0\|_{\shc^\beta}E_{1-\theta}(c\|g_w\|_{C_T\shc^{-\alpha}}T^{1-\theta}\Gamma(1-\theta)), \end{equation}
where $E_{1-\theta}$ is the Mittag-Leffler function
defined in \eqref{eq:Mittag}.
By pointwise product property \eqref{eq:bony}
and from  bound \eqref{eq:fg-L1} of Lemma \ref{lm:prep0} with $g\equiv 1$ and $f=w(t)$, for all $t\in[0,T]$, we have
\begin{equation*}
\|g_w(t)\|_{\shc^{-\alpha}}
 \leq \|F(K*w(t))\|_{\shc^\beta} \|b(t)\|_{\shc^{-\alpha}} 
 \leq C_{F,K} (1+ \Vert w(t)\Vert_{L^1}) \|b(t)\|_{\shc^{-\alpha}}  
 \le 2  C_{F,K}\|b\|_{C_T\shc^{-\alpha}}.
\end{equation*}
Therefore, taking the supremum over $[0,T]$ we get
\begin{equation}\label{eq:FKwbound}
\|g_w\|_{C_T\shc^{-\alpha}}\leq C_{F,K} \|b\|_{C_T\shc^{-\alpha}}.
\end{equation}
 Thus, inserting \eqref{eq:FKwbound} in \eqref{eq:boundtau} one concludes
 \begin{equation*}
   \|\tau(w)\|_{C_T\shc^\beta}\leq M_0(\|b\|_{C_T\shc^{-\alpha}}),
\end{equation*}
where
$$ M_0(x) =    c\|v_0\|_{\shc^\beta}E_{1-\theta}(C_{F,K} x T^{1-\theta}\Gamma(1-\theta)),
$$
which concludes the proof.   \qedhere
\end{proof}

We remark that, for every $M > 0$, the ball $B^M$ defined in \eqref{eq:BMtau}
is a closed set (therefore an $F$-space) even
with respect to any equivalent distance $ d_{\rho,\be} $, for some $\rho\geq 0$, see Definition \ref{def:rhonorm}.
Let  now $M \ge  M_0(\|b\|_{C_T\shc^{-\alpha}})$,
where $M_0$
is given in Lemma \ref{lm:boundtau}.
Then the mapping $\tau$
defined in  \eqref{eq:linsolmap}, maps obviously $B^M$ into $B^M$,
taking into account Lemma \ref{lm:boundtau}.
With a slight abuse of notation we denote the restriction of  $\tau$ to $B^M$
 by the same letter and we prove below that it is
a contraction under $d_{\rho,\be}$ for some $\rho> 0$.


\begin{proposition} \label{pr:boundtau}
Let Assumptions \ref{ass:beta}, \ref{ass:K} and \ref{ass:F} hold.
Let $v_0 \in \shc^{\beta+}  \cap L^1$
non-negative, such that $\Vert v_0 \Vert_{L^1} \le 1$.
Let also $0 < \alpha < \beta$ such that $b \in C_T\shc^{-\alpha}.$
  Let $M \ge M_0(\|b\|_{C_T\shc^{-\alpha}})$, where $M_0$ is provided by Lemma \ref{lm:boundtau}.
 Then there exists $\rho > 0$, depending on 
  $M, C_{F,K},T, \|b\|_{C_T\shc^{-\al}},\be,\al$, such that
\begin{equation} \label{eq:drho}
  d_{\rho,\be}(\tau(w),\tau(z))\leq \frac{1}{2} d_{\rho,\be}(w,z), \forall w, z
  \in B^M,
  \end{equation}
  where $d_{\rho,\beta}$ was defined in \eqref{rho_dist}.
\end{proposition}

\begin{proof}
Let $w, z \in B^M$ and
   define $v:=\tau(w)$, $y=\tau(z)$. We recall the notation  $g_w$ defined in \eqref{eq:gw}.
  Let us consider the real positive number $\theta:=\frac{\al+\be+1}{2}$. Now using \eqref{eq:sben} with $v_n=y, g=g_w, g_n=g_z$ and $v_0^n = v_0$, for every $t \in [0,T]$,
  one gets
 \begin{equation}
  \|v(t)-y(t)\|_{\shc^\be}\leq c\int_0^t(t-s)^{-\theta}\left[\|v(s)\|_{\shc^{\be}}\|g_w(s)-g_z(s)\|_{\shc^{-\al}}+\|y(s)-v(s)\|_{\shc^\be}\|g_z(s)\|_{\shc^{-\al}}\right]ds. 
 \end{equation}
 Now we multiply  both sides by $e^{-\rho t}$ and taking into account that $v\in B^M$ together with \eqref{eq:bony} and \eqref{eq:bonyt}, we have
\begin{align} \label{eq:first-est}
   \notag   e^{-\rho t} &\|v(t)-y(t)\|_{\shc^\be}\\
  \notag    \leq & c\int_0^t(t-s)^{-\theta}e^{-\rho(t-s)}e^{-\rho s}\left[\|v(s)\|_{\shc^\be}\|g_w(s)-g_z(s)\|_{\shc^{-\al}}+\|y(s)-v(s)\|_{\shc^\be}\|g_z(s)\|_{\shc^{-\al}}\right]ds\\
  \notag   \leq &c\left[M\sup_{s\in[0,T]} e^{-\rho s} \|g_w(s)-g_z(s)\|_{\shc^{-\al}}+\|y-v\|^{(\rho)}_{C_T\shc^\be}\|g_z\|_{C_T\shc^{-\al}}\right]\int_0^t(t-s)^{-\theta}e^{-\rho(t-s)}ds\\ 
      \leq & c\Big[M \sup_{s\in[0,T]} e^{-\rho s} \|F(K*w(s))-F(K*z(s))\|_{\shc^{\be}}\|b(s)\|_{\shc^{-\al}}\\
  \notag &+\|y-v\|^{(\rho)}_{C_T\shc^\be}\|F(K*z)\|_{C_T\shc^{\be}}\|b\|_{C_T\shc^{-\al}}\Big]\rho^{\theta-1}\Gamma(1-\theta), \ t \in [0,T].
\end{align}
We now bound the two terms involving $K$. Using bound \eqref{eq:fg-L1} of Lemma \ref{lm:prep0} with $\g=\be$, $g=1$ and $f=z(t)$ with any fixed $t\in[0,T]$ 
we get
\begin{equation*}
  \|F(K*z(t))\|_{\shc^{\be}} \le C_{F,K} (1 + \Vert z(t) \Vert_{L^1}) \le C_{F,K} (1 + \sup_{t\in[0,T]} \Vert z(t) \Vert_{L^1})\le 2 C_{F,K},
\end{equation*}
so that taking the sup over $t$ we have
\begin{equation} \label{eq:second-est}
 \|F(K*z)\|_{C_T\shc^\be}\le 2 C_{F,K}.
\end{equation}
Using \eqref{eq:FC_al} in Lemma \ref{lm:prep2} with $\g=\be$, $f=w(s)$ and $g=z(s)$, for
$s \in [0,T] $, we have
\begin{equation*}
 \|F(K*w(s))-F(K*z(s))\|_{\shc^{\be}}  \le    C_{F,K} (1 + \Vert w(s) + z(s) \Vert_{L^1}) \Vert w(s) - z(s) \Vert_{ \shc^\be} \le 3 C_{F,K} \Vert w(s) - z(s) \Vert_{ \shc^\be},
\end{equation*}
  and multiplying both sides by $e^{-\rho s}$ and taking the sup over $s\in[0,T]$ we get 
  \begin{equation} \label{eq:third-est}
 \sup_{s\in[0,T]} {e^{-\rho s}}  \|F(K*w(s))-F(K*z(s))\|_{\shc^\be}  \le  3 C_{F,K} \Vert w- z \Vert^{(\rho)}_{C_T \shc^\be}.
 \end{equation}
Inserting \eqref{eq:second-est} and \eqref{eq:third-est} into \eqref{eq:first-est} and
taking the supremum over the set $[0,T]$ we obtain
\begin{align*}
  \|y-v\|^{(\rho)}_{C_T\shc^\be} &\leq  \rho^{\theta-1}\Gamma(1-\theta) \left[3 c M C_{F,K}   \|b\|_{C_T\shc^{-\al}}  \Vert w - z \Vert^{(\rho)}_{C_T \shc^\be} + 2 c C_{F,K} \|y-v\|^{(\rho)}_{C_T\shc^\be} \|b\|_{C_T\shc^{-\al}}\right].
\end{align*}
Rearranging the term involving $\|y-v\|^{(\rho)}_{C_T\shc^\be}$ on the right hand side we get
\begin{align*}
	\|y-v\|^{(\rho)}_{C_T\shc^\be}\left[1- \rho^{\theta-1}\Gamma(1-\theta)2cC_{F,K} \|b\|_{C_T\shc^{-\al}}\right] &\leq  \rho^{\theta-1}\Gamma(1-\theta) \left[3 c M C_{F,K}   \|b\|_{C_T\shc^{-\al}}  \Vert w - z \Vert^{(\rho)}_{C_T \shc^\be} \right].
\end{align*}
Now we choose $\rho$ such that $1- \rho^{\theta-1}\Gamma(1-\theta)2cC_{F,K} \|b\|_{C_T\shc^{-\al}}> 0$ and  we can divide and we get 

\begin{equation*}
	\|y-v\|^{(\rho)}_{C_T\shc^\be} \leq \frac{\rho^{\theta-1}\Gamma(1-\theta) 3 c  C_{F,K}  \|b\|_{C_T\shc^{-\al}}}{1- \rho^{\theta-1}\Gamma(1-\theta)2cC_{F,K} \|b\|_{C_T\shc^{-\al}}}M  \Vert w - z \Vert^{(\rho)}_{C_T \shc^\be}. 
\end{equation*}
We now can choose $\rho $ big enough so that 
\begin{equation*}
\frac{\rho^{\theta-1}\Gamma(1-\theta) 3 c  C_{F,K}  \|b\|_{C_T\shc^{-\al}}}{1- \rho^{\theta-1}\Gamma(1-\theta)2cC_{F,K} \|b\|_{C_T\shc^{-\al}}}\leq\frac1{2M}
\end{equation*}
and this guarantees the validity of \eqref{eq:drho}. 
\medskip
%
\end{proof}

\begin{theorem}\label{thm:WEFPPDE}
   Let Assumptions \ref{ass:beta}, \ref{ass:K} and \ref{ass:F} hold true.
  Let $v_0 \in \shc^{\beta+}$ be non-negative function such that $\Vert v_0 \Vert_{L^1} \le 1.$ 
Then there exists a (unique) solution in $C_T \shc^{\beta} \cap S_{L^1}$ to the non-local singular non-linear Fokker-Planck PDE (\ref{eq:FPlim}).
\end{theorem}
\begin{proof}
  Let $\alpha \in (0, \beta)$ such that $b \in C_T\shc^{-\alpha}$. 
   Let $M_0 := M_0(\|b\|_{C_T\shc^{-\alpha}})$, where $M_0$ is provided by Lemma \ref{lm:boundtau},
  so that $\tau: B^{M_0} \mapsto B^{M_0}$. By Proposition \ref{pr:boundtau}
  $\tau$ is a contraction with respect to $d_{\rho,\beta}$
  for some suitable $\rho$. By the Banach  fixed point
  theorem $\tau$ admits a (unique) fixed point $v \in B^{M_0}$.
  By definition of $\tau$,  for $t \in [0,T]$,  we write 
  $$ v(t)= P_{t}v_0-\int_0^t P_{t-s}\left({\rm div}\left(v(s)\cdot g_v (s)
    \right)\right)ds =
   P_{t}v_0-\int_0^t P_{t-s}\left({\rm div}\left(v(s) \cdot (F(K*v(s))) \cdot  b(s) \right)\right)ds, $$
  taking into account the associativity property
  $$ v(s) g_v(s) =  v(s) \cdot (F(K*v(s))) \cdot  b(s)) =
  (v(s) F(K*v(s))) \cdot  b(s), s \in [0,T], $$
  see Remark \ref{rmk:associativity} 2. and 3., 
 which establishes existence.
Uniqueness is a particular case of Proposition \ref{pr:uniFPlim}.

 \end{proof}
\begin{remark}\label{rkm:WEFPPDE}
  Let $\alpha < \beta$ such that $b \in C_T\shc^{-\alpha}$.
  A side-effect of previous proof is that the unique
  aforementioned solution belongs to $B^{M_0(\|b\|_{C_T\shc^{-\alpha}})}$,
  where $M_0$ is the function defined in Lemma \ref{lm:boundtau}.
\end{remark}

\subsection{Continuity of the solutions with respect to the drift}

\label{sc:drift}

We end this  section by investigating the continuity 
of a solution to the non-local singular non-linear Fokker-Planck equation
with respect to the $b$ component of the drift.
Given a sequence $(b^n)$ converging to $b$ in $C_T\shc^{(-\beta)+}$,
Theorem \ref{thm:WEFPPDE} states the existence of
unique solutions $(v^n)$ to
\begin{equation}\label{eq:PDEdm}
   \left\{
  \begin{array}{l}
      \partial_t v^n=\frac12\Delta v^n-{\rm div}(v^nF(K*v^n)b^n)\\
      v^n(0)=v_0\in\shc^{\be+} .    
  \end{array}  
  \right. 
\end{equation} 
A typical example
is given by $b^n(t):=p_{\frac{1}{n}}*b(t)=P_{\frac{1}{n}}b(t),   t \in [0,T]$,
so that 
$b$ is smoothed through the convolution with the heat semigroup.

\begin{proposition}\label{pr:PDEdmcont}
	\begin{enumerate}
        \item[(i)]  Let $0<\al<\be<1/2$ such that for $i =1,2$,
          $b^i\in C_T\shc^{-\al}$.
                  Let $v_0 \in \shc^{\be+}$, such that   $v_0$ is a non-negative function such that $\Vert v_0 \Vert_{L^1} \le 1.$
		For $i = 1,2$, let
                $v^i$
                be the (mild) solution of (\ref{eq:FPlim}) with $b=b^i$.
		Then there exists a increasing function $\ell:\R^+ \rightarrow \R^+$, depending
		$\beta, \alpha, T, C_{F, K}, \Vert v_0\Vert_{\shc^\beta}$,
                such that 
		\begin{equation}\label{eq:ellGronwall}
			\|v^1-v^2\|_{C_T \shc^{\al}}\leq \ell(\|b^1\|_{C_T\shc^{-\alpha}}  \lor\|b^2\|_{C_T\shc^{-\alpha}})\|b^1-b^2\|_{C_T\shc^{-\alpha}},
		\end{equation}
		for all $t\in[0,T]$.
		\item[(ii)] Let $(b^n)$ be a sequence in $C_T\shc^{(-\be)+}$ and $v^n$ be the solution of (\ref{eq:PDEdm}) with $b=b^n$ and $v$ be the solution to (\ref{eq:FPlim}). If $b^n\to b$ in $C_T\shc^{(-\be)+}$ then $v^n \to v $ in $C_T\shc^\be$.
	\end{enumerate}
\end{proposition}

\begin{remark} 
  Above we refer to  the solutions
  belonging to $C_T\shc^\be\cap S_{L^1}$, see Theorem \ref{thm:WEFPPDE}.
\end{remark}

    \begin{proof}[Proof of Proposition \ref{pr:PDEdmcont}] 
      {\it (i)}
By Schauder and Bernstein inequalities (Lemmata  \ref{lm:schauder} and \ref{lm:bern}) and recalling the notation \eqref{def:phi}, we get
      \begin{align} \label{eq:intBS}
           \notag   \|v^1(t)-v^2(t)\|_{\shc^\be}\leq&\left\|\int_0^t P_{t-s}\text{div}[\phi(v^1(s))b^1(s)-\phi(v^2(s))b^2(s)]ds\right\|_{\shc^\be}\\
           \notag   \leq&c\int_0^t (t-s)^{-\frac{\al+\be+1}{2}}\left\|\text{div}[\phi(v^1(s))b^1(s)-\phi(v^2(s))b^2(s)]\right\|_{\shc^{-\al-1}}ds\\
               \leq&c\int_0^t (t-s)^{-\frac{\al+\be+1}{2}}\left\|\phi(v^1(s))b^1(s)-\phi(v^2(s))b^2(s)\right\|_{\shc^{-\al}}ds.
      \end{align}
         In order to bound the term in the $\shc^{-\al}$-norm inside previous integral we make use of the pointwise product (\ref{eq:bony}), the Lipschitz and linear growth property of $\phi$ contained
         in \eqref{eq:preplin} and \eqref{eq:lipgrow} of Proposition \ref{pr:C_al}, with $\gamma = \beta, f = v^1, g= v^2$. 
For $s\in[0,T]$,  we get         
         \begin{align*}
\left\|\phi(v^1(s))b^1(s)-\phi(v^2(s))b^2(s)\right\|_{\shc^{-\al}}=&\left\|\phi(v^1(s))b^1(s)+\phi(v^1(s))b^2(s)-\phi(v^1(s))b^2(s)-\phi(v^2(s))b^2(s)\right\|_{\shc^{-\al}}\\  
\leq&  \left\|\phi(v^1(s))b^1(s)-\phi(v^1(s))b^2(s)\right\|_{\shc^{-\al}}+\left\|\phi(v^1(s))b^2(s)-\phi(v^2(s))b^2(s)\right\|_{\shc^{-\al}}\\
\leq&\|\phi(v^1(s))\|_{\shc^{\be}}\|b^1(s)-b^2(s)\|_{\shc^{-\al}}+\|\phi(v^1(s))-\phi(v^2(s))\|_{\shc^{\be}}\|b^2(s)\|_{\shc^{-\al}}\\
\leq& C_{F,K}[\|v^1\|_{C_T\shc^{\be}}(1 + \|v^1\|_{B_T L^1})\|b^1-b^2\|_{C_T\shc^{-\al}} \\  
&+ \|b^2\|_{C_T\shc^{-\al}}(1+\|v^1\|_{C_T\shc^{\be}}+\|v^2\|_{C_T\shc^{\be}}) (1 + \|v^1 + v^2\|_{B_T L^1})  \|v^1(s)-v^2(s)\|_{\shc^{\be}}]\\
\leq&  C_{F,K}[ \|v^1\|_{C_T\shc^{\be}} \|b^1-b^2\|_{C_T\shc^{-\al}}  \\
      &+ \|b^2\|_{C_T\shc^{-\al}}(1+\|v^1\|_{C_T\shc^{\be}}+\|v^2\|_{C_T\shc^{\be}})  \|v^1(s)-v^2(s)\|_{\shc^{\be}}],
         \end{align*}
         taking into account that $v^1, v^2$ belong to $S_{L^1}$.
         Let $M_0$
         be the function coming from Lemma \ref{lm:boundtau}. 
Owing to Remark \ref{rkm:WEFPPDE}, previous inequalities yield
         \begin{align} \label{eq:ell12}
           \notag   \left\|\phi(v^1(s))b^1(s)-\phi(v^2(s))b^2(s)\right\|_{\shc^{-\al}}  \le& C_{F,K}\left(M_0(\Vert b^1\Vert_{C_T \shc^{-\al}} \vee \Vert b^2\Vert_{C_T \shc^{-\al}}) \|b^1-b^2\|_{C_T\shc^{-\al}} \right.\\
      \notag     &+ \left. \|b^2\|_{C_T\shc^{-\al}}
(1 + 2 M_0(\Vert b^1\Vert_{C_T \shc^{-\al}} \vee \Vert b^2\Vert_{C_T \shc^{-\al}})) 
                      \|v^1(s)-v^2(s)\|_{\shc^{\be}}\right) \\
           \notag  =&\ell_1(\Vert b^1\Vert_{C_T \shc^{-\al}} \vee \Vert b^2\Vert_{C_T \shc^{-\al}}))  \|b^1-b^2\|_{C_T\shc^{-\al}} \\
           &+ \ell_2(\Vert b^1\Vert_{C_T \shc^{-\al}} \vee \Vert b^2\Vert_{C_T\shc^{-\al}})  \|v^1(s)-v^2(s)\|_{\shc^{\be}},    
             \end{align}
  where
  $$\ell_1(x) := C_{F,K} M_0(x), \ \ell_2(x) := C_{F,K} x (1 + 2 M_0(x)),
  \ x \ge 0,$$
  are two increasing functions.
  Inserting \eqref{eq:ell12} into \eqref{eq:intBS}, keeping the notation used previously for $1-\theta=\frac{1-\al-\be}{2}$,  for all $t \in [0,T]$, gives
\begin{align*}
  \|v^1(t)-v^2(t)\|_{\shc^{\al}}\leq& 
               c  \frac{T^{1-\theta}}{1-\theta}   \ell_1(\Vert b^1\Vert_{C_T \shc^{-\al}} \vee \Vert b^2\Vert_{C_T \shc^{-\al}})  \|b^1-b^2\|_{C_T\shc^{-\al}} \\
&+ c \ell_2(\Vert b^1\Vert_{C_T \shc^{-\al}} \vee \Vert b^2\Vert_{C_T \shc^{-\al}}) \int_0^t (t-s)^{-\frac{\al+\be+1}{2}}  \|v^1(s)-v^2(s)\|_{\shc^{\be}} ds.  
  \end{align*}
 At this point, defining
$$ \ell(x) :=  c  \frac{T^{1-\theta}}{1-\theta}   \ell_1(x) E_{1-\theta}(c\ell_2(x)), x \ge 0,$$
the previous inequality finally establishes  \eqref{eq:ellGronwall},
making use of the  generalised Gronwall's inequality \cite{ye}.

{\it (ii)} Let $v^n$ (resp.\ $v$)    be the  solution of (\ref{eq:PDEdm}) (resp.\ \eqref{eq:FPlim})
given by Theorem \ref{thm:WEFPPDE}.
Let $\al<\be$, such that $b^n\to b$ in $C_T\shc^{-\al}$. We can now apply {\it item  (i)}  to $v^1:= v$ and $v^2:=v^n$ to obtain
        \begin{equation}\label{continuity_seq_PDE}
          \|v-v^n\|_{C_T\shc^\be}\leq \ell(\|b\|_{C_T\shc^{-\al}}\lor\|b^n\|_{C_T\shc^{-\al}})\|b-b^n\|_{C_T\shc^{-\al}}.
        \end{equation}
        Now we observe that $\sup_n\|b^n\|_{C_T\shc^{-\al}}<\infty$ because $b^n\to b$ in $C_T\shc^{-\al}$, and so 
        \begin{equation*}
            \ell(\|b\|_{C_T\shc^{-\al}}\lor\|b^n\|_{C_T\shc^{-\al}})\leq \ell(\|b\|_{C_T\shc^{-\al}}\lor\sup_n\|b^n\|_{C_T\shc^{-\al}}),
        \end{equation*}
        since $\ell$ is increasing. Therefore
        \begin{equation*}
          \|v-v^n\|_{C_T\shc^{\be}}\leq \ell(\|b\|_{C_T\shc^{-\al}}\lor\sup_n\|b^n\|_{C_T\shc^{-\al}})
          \|b-b^n\|_{C_T\shc^{-\al}}\to 0,\text{ as } n \to \infty.
        \end{equation*}
       Finally this implies  the convergence $v^n\to v$ in $C_T\shc^{\be}$. 
    \end{proof}


    \section{An associated singular McKean stochastic differential equation}

\label{sec:PRFP}

In this section we still suppose the validity of Assumptions \ref{ass:beta},
\ref{ass:K} and \ref{ass:F}.
Here $v_0 \in \shc^{\beta +}$ is a non-negative function such that $\Vert v_0 \Vert_{L^1} = 1.$
The aim here is to study well-posedness
of an McKean SDE with singular drift and non-local law dependence
of the type
 \begin{equation}\label{eq:MKSDE}
 \left\{
 \begin{array}{l}
 dX_t=F(K*v(t,X_t))b(t,X_t)dt+dW_t\\
 v(t,\cdot) \text{ is the law density of }X_t\\
   X_0\sim v_0.
 \end{array}
 \right.
 \end{equation} 
The basic tool employed is 
the solution  to  the non-local singular non-linear Fokker-Planck PDE \eqref{eq:FPlim}: in particular we will construct a stochastic process whose
marginal law densities
are solutions to the PDE \eqref{eq:FPlim}.

Since the McKean SDE  \eqref{eq:MKSDE} is singular due to the presence of the distributional term $b\in C_T\shc^{(-\be)+}$, we first have to explain how to interpret  equation \eqref{eq:MKSDE}, which is only formal at this stage.
If $v$ in the first line of \eqref{eq:MKSDE} were given, this equation would be a non-McKean SDE with distributional drift, say $B$, like in \cite{flandoli_et.al14, issoglio_russoMPb}.
The modern formulation of  singular SDEs with distributional drift $B$ is done via the notion of {\em rough martingale problem}, which will be recalled in Section \ref{sc:rMP}.
A solution to the McKean SDE  \eqref{eq:MKSDE} will  be a solution of a non-McKean SDE (via rough martingale problem) with  distributional drift $B = F(K\ast v) b$, where $(v(t,\cdot))$ are the law densities of the unknown process. The well-posedness of the singular McKean SDE \eqref{eq:MKSDE}  will be studied in Section \ref{ssc:MK}.


\subsection{The rough martingale problem}\label{sc:rMP}

We start by recalling some known facts about the rough martingale problem, which is a way to give a mathematical meaning to SDEs with distributional drifts of the form
 \begin{equation}\label{eq:dSDE}
  \left\{
 \begin{array}{l}
 dX_t=B(t,X_t)dt+dW_t\\
 X _0\sim \nu,
 \end{array}
 \right.
 \end{equation}
 where $(W_t)$ is a $d$-dimensional Brownian motion,  and $B\in C_T\shc^{(-\be)+}$ and  $\nu\in\shp(\R^d)$ are 
given. 
 We refer the reader to \cite{flandoli_et.al14, issoglio_russoMPb} for more details.

Let us explain the idea behind the rough martingale problem by starting from the classical Stroock-Varadhan martingale problem, see Chapter 6 in \cite{stroock_varadhan}, which is studied when $B$ is a locally bounded  Borel  function.  In this case, if one takes $f\in C^{1,2}([0,T]\times \R^d)$, an application of It\^o's formula shows that the process
  $$M^f_t:= f(t,X_t) - f(0,X_0) - \int_0^t \shl f(r,X_r) dr, t \in [0,T]$$
    is a local martingale, where $\shl$ is the generator of a
    solution  $(X_t)$  to \eqref{eq:dSDE}, which is given by
 \begin{equation}\label{eq:L}
    \shl f:=\partial_t f+\frac12\Delta f+\nabla f \cdot B.
    \end{equation} 
   The space $  C^{1,2}([0,T]\times \R^d)$ is called domain of test functions for the Stroock-Varadhan martingale problem.
   
        When $B$ is a distribution living in the space $C_T\shc^{(-\be)+}$, we need to define a different domain of test functions because in general $\shl f $ is still a distribution  and thus the integral term $ \int_0^t \shl f(r,X_r) dr$ is a priori ill-posed. It turns out that a possible domain is constituted by the set of solutions $f$ to the
    Kolmogorov PDE
\begin{equation} \label{eq:PDE}
  \shl f = g, \ f(T,\cdot) = f_T,
\end{equation}
  for a  suitable class of functions $g, f_T$. This is illustrated in  Definition  \ref{def:domains} below, which is taken from \cite{issoglio_russoMPb}. This definition makes use of the weak solution  of the singular Kolmogorov PDE \eqref{eq:PDE} and its well-posedness and can be found in \cite{issoglio_russoPDEa}. They are recalled below for ease of reading.

\begin{definition}\label{def:weak} 
	Let  $g \in C_T  \shc^{(-\beta)+}  $, $f_T \in  \shc^{(1+\beta)+}$.
	We say that $f  \in C_T  \shc^{(1+\beta)+}$ is a weak solution of \eqref{eq:PDE} if, for all $\varphi \in \mathcal S(\R^d)$,  $f$ satisfies
	\begin{align}\label{eq:weak}
		&\langle  f_T,\varphi\rangle -\langle f(t),\varphi\rangle +\int_t^T \langle f(s) ,\frac12\Delta \varphi \rangle ds+ \int_t^T \langle \nabla f(s) B(s),\varphi\rangle ds    
		= \int_t^T\langle g(s),\varphi\rangle  d s, 
	\end{align}
	for all $t\in[0,T]$.
\end{definition}
\begin{remark} \label{rmk:Interp}
  In the second integral on the right-hand side involves the evaluation of the
  distribution  $\nabla f(s) B(s)$ which is a pointwise product and the test function  $\varphi$.
  \end{remark}

The following result is a consequence of Theorem 4.7 (ii) and Remark 4.8, together with Proposition 4.5 of \cite{issoglio_russoPDEa},
the latter one stating the equivalence between mild and weak solutions
of \eqref{eq:PDE}.

\begin{proposition} \label{prop:EU}
  Let  $g \in C_T  \shc^{(-\beta)+}  $, $f_T \in  \shc^{(1+\beta)+}$.
  There exists a unique weak solution $ f  \in C_T  \shc^{(1+\beta)+}$ of \eqref{eq:PDE}.
 \end{proposition}
To define the domain of the rough martingale problem, we make use of spaces of functions defined as the closures of compactly supported functions, ensuring that they are separable Banach spaces.
As defined in Section 2 of \cite{issoglio_russoMPb}, for any $\gamma \in \mathbb{R}$, we denote by $\bar{\mathcal{C}}_c^\gamma = \bar{\mathcal{C}}_c^\gamma (\mathbb{R}^d)$ the closure of the space of compactly supported functions $\mathcal{C}_c^\gamma$ in $\mathcal{C}^\gamma$, defined explicitly as

\[\bar{\mathcal{C}}_c^\gamma := \{f \in \mathcal{C}^\gamma : \exists (f_n) \subset \mathcal{C}_c^\gamma \text{ such that } f_n \to f \text{ in } \mathcal{C}^\gamma\}.
\]
Analogously to \eqref{eq:indBesov},
we denote the corresponding inductive spaces  
\[
  \bar {\mathcal C}_c^{\gamma+}:= \bigcup_{\alpha >\gamma} \bar{\mathcal C}_c^{\alpha} .
\]

 \begin{definition}\label{def:domains}
We define the domain 
   \begin{equation*}
\mathcal D_{\mathcal L} : =  \{ f \in C_T\mathcal C^{(1+\beta)+}:   \exists g \in  C_T\bar {\mathcal C}_c^{0+}  \text{ such that $f$ is a weak solution of \eqref{eq:PDE} with } f_T \in \bar{\mathcal C}_c^{(1+\beta)+  }\}.
\end{equation*}
\end{definition}
We can now give the definition of rough martingale problem.
    \begin{definition}\label{def:rMP}
     Let $\Omega_T^d$ be the canonical space  $C([0,T];\R^d)$ equipped with
     its Borel $\sigma$-field $\shf_T$ and $Z$ be the canonical process. A probability measure $\P\in\shp(\Omega_T)$ is a solution to the rough martingale problem with distributional drift $B$ and initial condition
     $\nu$ 
      if and only if for every $f\in\shd_\shl$ 
    $$f(t,Z_t)-f(0,Z_0)-\int_0^t\mathcal{L}f(s,Z_s)ds$$
    is a local martingale under $\P$, and $\P\circ Z_0^{-1}=\nu$, i.e. $\mathcal{L}aw(Z_0)=\nu$. 
We will also say that $\P$ is a solution to rMP$(B,\nu)$.
  \end{definition}

 \begin{remark} \label{rmk:SVMP}
     Lemma 4.1 of \cite{issoglio_russoMPb}
    states the following.
If $B \in C_T  \shc^{0+}$ then $\P$ is a solution to rMP$(B,\nu)$
if and only if $\P$ is a solution to the Stroock-Varadhan martingale problem.
    \end{remark}
    
We proceed by recalling Theorem 4.5 in \cite{issoglio_russoMPb}, which states  the following.
\begin{proposition} \label{prop:MPlin} 
Let $B \in C_T \shc^{(-\beta)+}$
  and let $\nu$ be  a probability measure.
  Then the rough martingale problem rMP$(B,\nu)$  admits existence and uniqueness in the sense of Definition \ref{def:rMP}.
\end{proposition}
At this point we state and prove the probabilistic representation  for a solution
to the linear  Fokker-Planck PDE \eqref{eq:FPlin} with $g = B$ and initial condition $v_0$ with unitary mass, namely for 
\begin{equation}\label{eq:FPlimBis}
    \left\{
     \begin{array}{l}
     \partial_t v=\frac12\Delta v-{\rm div}(vB),\\
      v(0)=v_0.
     \end{array}
      \right.
\end{equation}
Well-posedness for this PDE is  provided by Theorem \ref{thm:exunlinFP}
  with $g = B$.

Below we show that given a solution $v$ to the Fokker-Planck PDE \eqref{eq:FPlimBis}, we can find a probability measure  on  the canonical space $\Omega^d_T$ (or equivalently a stochastic process on the canonical space) that admits  $v(t), t \in [0,T],$ as the density of its time marginals. 

\begin{proposition}\label{prop:SDEtoPDE}
  Let $v_0$ be a probability density which belongs to $ \shc^{\beta +}$ and $B\in C_T \shc^{(-\be)+}$.
Let $\P$ be the solution of rMP($B,v_0(x)dx)$ given by Proposition  \ref{prop:MPlin}.
Then, for all $t\in [0,T]$, the time marginals of $\P$ admit a density $v(t)$.   Furthermore, $v$ is the  solution to the linear Fokker-Planck PDE \eqref{eq:FPlimBis}, provided by Theorem \ref{thm:exunlinFP}, with $g = B$.
\end{proposition}

\begin{proof} 
 Let $v  \in C_T \shc^{\be}$ be the solution of the Fokker-Planck PDE 
 \eqref{eq:FPlimBis}, see Theorem \ref{thm:exunlinFP}
 with $g = B$.
 Let $B^n$ so that $B^n(t,\cdot)= p_{\frac{1}{n}} \ast B(t,\cdot)$
  and  $v_n$ be the solution of  \eqref{eq:FPlin} with 
  $g = B^n$.
From Lemma \ref{lm:Xncont} with $g=B,g^n=B^n$ we have that $B^n$ converges to $B$ in $C_T \shc^{(-\beta)+}$, so Theorem \ref{thm:exunlinFPCont} guarantees that
 $v^n$ converges to $v$ in $C_T \shc^{\be}$.
  By Proposition \ref{lin_sing_PDE}
  $v^{n}$ is non-negative and $v^{n}\in L^1(\R^d)$, for every $n\in\N$.
  Thanks to Theorem 2.6 in \cite{figalli}, there exists a solution $\mathbb P^{n}$ to the Stroock-Varadhan martingale problem with
drift $B^n$ and initial condition $\nu(dx) = v_0(x)dx$
whose density at time $t$ is exactly $v^{n}(t)$. 
Owing to Remark \ref{rmk:SVMP}
the measure $\P^{n}$ coincides with the unique solution to rMP$(B^n, v_0(x)dx)$.
Since $B^n$ converges to $B$ in $C_T \shc^{(-\beta)+}$, it also converges in $C_T\shc^{-\be}$.
Therefore, by Theorem 4.3
in \cite{issoglio_russoMPb}
$\P^{n}$ converges weakly to $\P$.
Consequently 
$v^{n}(t,x)d x$ converges weakly to the time marginal  of $\P$, for any $t\in [0,T]$.
This implies that $v(t)$ is the density of the time marginal   $\P_{Z_t}$, which proves the result.
\end{proof} 

\begin{corollary} \label{cor:mass-con} Suppose $B \in C_T \shc^{-\beta}$ and $v_0 \in \shc^{\beta+}$.
  Let $v$ be the solution  to the Fokker-Planck PDE  \eqref{eq:FPlimBis}
  such that $v_0$ is non-negative $\|v_0\|_{L^1}=1$.
    Then $v$ preserves non-negativity and conserves the mass in the sense that
  $\|v(t)\|_{L^1}=1$ for each $t\in[0,T]$.

\end{corollary}
\begin{proof}
  Let $\P$ be provided by Proposition \ref{prop:MPlin}.
  By Proposition \ref{prop:SDEtoPDE}, the marginals
  admit a density, which is given by  $w(t,\cdot),  t \in [0,T]$, where
  $w$ is the solution to the PDE  \eqref{eq:FPlimBis}.
  In particular $w$ is non-negative and conserves the mass.
The uniqueness property of the PDE \eqref{eq:FPlimBis}
  allows to conclude.
    \end{proof}

\begin{remark} \label{rmk:mass-con}
The statement of previous corollary  is stronger than the result obtained in Proposition \ref{lin_sing_PDE}, because even assuming $\|v_0\|_{L^1}=1$ therein, 
we could only conclude that $\|v(t)\|_{L^1}\leq 1$ for each $t\in[0,T]$.
\end{remark}

\subsection{Well-posedness of the singular McKean SDE}\label{ssc:MK}

We consider here the McKean SDE formally given by   \eqref{eq:MKSDE}, for which a mathematical meaning is given  via a suitable rough martingale problem, as detailed below.  
\begin{definition}\label{def:MKrMP}
  A solution to  the McKean  rough martingale problem
  with  drift $ F(K*\,\cdot\,)b$ and initial condition $v_0$,   is  a Borel  probability measure $\P$ on  $\Omega^d_T$,
  such that the following hold.
\begin{itemize}
\item[-] The time marginals of $\P$ admit a density for each time $t\in[0,T]$, denoted by $v(t)$;
\item[-] $v: [0,T ]\times\R^d  \to \R$  is an element of $C_T \mathcal C^{\be}$;
\item[-] $\P$ solves the rough martingale problem rMP$(B,\nu)$ according to Definition \ref{def:rMP} with
 $B:=F(K*v)b$ and $\nu(dx) = v_0(x) dx$. 
\end{itemize}
We denote it by McKean--rMP$(F(K*\, \cdot \,)b,v_0)$ for shortness\footnote{The notation $F(K*\, \cdot \,)b$ is due to the fact that the density  $v$ is unknown, and thus it is not part of the given coefficients.}.
\end{definition}

\begin{theorem}\label{thm:WPMKSDE}
  Let Assumptions \ref{ass:beta}, \ref{ass:K} and \ref{ass:F} hold.
    Let $v_0 \in \shc^{\beta +}$ non-negative function such that $\Vert v_0 \Vert_{L^1} = 1$. 
    Then there exists a unique solution $\P$ to the McKean-rMP($F(K*\,\cdot\,)b,v_0$).
\end{theorem}

\begin{proof}
	{\it (Existence).} 
	Let $v$ be the solution  of PDE \eqref{eq:FPlim} provided by Theorem \ref{thm:WEFPPDE} which belongs to $C_T \shc^{\beta}$.
By the associativity property stated in Remark \ref{rmk:associativity}
implies that $v$ is also a solution of \eqref{eq:FPlin}
with $g = g_v$, where $g_w$ was introduced in \eqref{eq:FPlin}
for any $w \in C_T\shc^{\beta}$.
        By Corollary \ref{lm:BNcont} with $g=1$ and $f=v$, we have $t \mapsto  F(K*v(t))   \in C_T \shc^\be$. 	From Assumption \ref{ass:beta} there is $\alpha <\beta$ such that $b \in C_T \shc^{-\al}$. Since $ \be-\al >0$, by
	\eqref{eq:bonyt}, $t \mapsto B(t):= F(K*v(t))b(t)$  belongs to $C_T \shc^{-\al}$, thus it belongs to $C_T \shc^{(-\beta)+}$.
	By Propositions \ref{prop:MPlin} and  \ref{prop:SDEtoPDE} there is a unique solution
        $\P$ of rMP($B,v_0(x)dx$) 
	such that  time marginal  $\P_{Z_t}$ has a  density which is precisely $v(t)$ for all $t\in[0,T]$.
	This is enough to prove existence.
	
	{\it (Uniqueness).} Let $\P^i$, for $i=1,2$, be two solutions of the McKean-rMP$(F(K*\,\cdot\,)b, v_0)$, with time marginals   $v^i \in C_T \shc^\be $.
	In particular they are solutions to rMP$(F(K*v^i)b, v_0(x)dx$).
	By Proposition \ref{prop:SDEtoPDE} for each $i=1,2$,  $v^i$ 
	satisfies the linear  Fokker-Planck PDE \eqref{eq:FPlimBis}, with $v=v^i$ and $B=F(K*v^i)b$. 
Again by Remark \ref{rmk:associativity},
for both $i=1,2$, $v^i$ is also a solution 
to  the non-linear PDE \eqref{eq:FPlim}. 
	Therefore by  uniqueness of the solution to PDE \eqref{eq:FPlim}, provided by Theorem \ref{thm:WEFPPDE},  
	we must have  $v^1=v^2$, which we denote by $v$.
	Consequently  $\P^i, i = 1,2$ are solutions to   rMP($ F(K*v)b, v_0(x)dx)$.
	By Proposition \ref{prop:MPlin} we finally conclude
	$\P^1 = \P^2$.
\end{proof}

{\bf Data availability statement.} The articles does not
contain data.

\medskip
  
{\bf Acknowledgements.} 
The author L.\ Bondi was partially supported by GNAMPA Progetti di Ricerca 2026 CUP: E53C25002010001.
The author E.\ Issoglio was partially supported by
PRIN-PNRR2022 (P20224TM7Z) CUP: D53D23018780001, from EU - Next Generation
EU - PRIN2022 (202277N5H9) CUP: D53D23005670006 and  GNAMPA Progetti di Ricerca 2026 CUP: E53C25002010001.
   The author F.\ Russo was partially supported by the  ANR-22-CE40-0015-01 (SDAIM).

\bibliographystyle{plain}
\bibliography{../BIBLIO_FILE/Biblio_Tesi_Luca}
\end{document}